\documentclass[12pt]{amsart}
\usepackage{amssymb, verbatim}
\usepackage{graphicx,color}
\usepackage{enumerate}
\usepackage{color,url}
\usepackage{cite}
\usepackage{mathtools}
\usepackage{xcolor}

\date{September 3, 2026}

\DeclareMathOperator\supp{supp}

\newtheoremstyle{ptheorem}{1em}{0em}{\itshape}{}{\bfseries}{.}{.5em}{\thmname{#1} \thmnumber{#2}\thmnote{ (\hspace{-.01pt}{#3})}}

\theoremstyle{ptheorem}
\newtheorem{theorem}{Theorem}
\newtheorem{corollary}[theorem]{Corollary}
\newtheorem{lemma}[theorem]{Lemma}
\newtheorem{proposition}[theorem]{Proposition}
\newtheoremstyle{hdef}{1em}{0em}{}{}{\bfseries}{.}{.5em}{\thmname{#1} \thmnumber{#2}\thmnote{ (\hspace{-.01pt}{#3})}}
\theoremstyle{hdef}
\newtheorem{definition}[theorem]{Definition}
\newtheorem{example}[theorem]{Example}
\theoremstyle{remark}
\newtheorem{remark}[theorem]{Remark}
\newtheorem{problem}{Problem}

\numberwithin{equation}{section}
\numberwithin{theorem}{section}

\title[Stieltjes-generated moment functionals]{Stieltjes-generated moment functionals and orthogonality: a transfer principle and an image theorem}

\author[I. Area]{Iv\'an Area}
\address[I. Area]{IFCAE, Universidade de Vigo, Departamento de Matem\'atica Aplicada II, E. E. Aeron\'au\-tica e do Espazo, Campus As Lagoas s/n, 32004 Ourense, Spain}
\email[I. Area]{area@uvigo.gal}

\author[F. J. Fern\'andez]{Francisco J. Fern\'andez}
\address[F. J. Fern\'andez]{Departamento de Estat\'{\i}stica, An\'alise Matem\'atica e Optimizaci\'on,  Universidade de Santiago de Compostela, 15782, Facultade de Matem\'aticas, Campus Vida, Santiago, Spain. CITMAga, Santiago de Compostela, 15782, Spain}
\email[F. J. Fern\'andez]{fjavier.fernandez@usc.es}

\author[F. A. F. Tojo]{F. Adri\'an F. Tojo}
\address[F. A. F. Tojo]{Departamento de Estat\'{\i}stica, An\'alise Matem\'atica e Optimizaci\'on,  Universidade de Santiago de Compostela, 15782, Facultade de Matem\'aticas, Campus Vida, Santiago, Spain. CITMAga, Santiago de Compostela, 15782, Spain}
\email[F. A. F. Tojo]{fernandoadrian.fernandez@usc.es}



\begin{document}

\begin{abstract}
We study the orthogonality induced by Stieltjes derivatives from the point of view of ordinary polynomial moment functionals.  If \(g\) is a derivator and the Stieltjes monomials \(g_{x_0,n}\) are linearly independent, the map \(\Phi(t^n)=g_{x_0,n}\) is an isomorphism of differential algebras onto the ordinary polynomial algebra, and the intrinsic \(\star\)-bilinear form on \(g\)-polynomials is the pull-back of a single ordinary moment functional \(\mathcal L\).  The object of interest is therefore not a new polynomial algebra but the functional generated by the derivator, and our main theorem determines exactly which functionals occur.  Let \(c_-\) and \(c_+\) be the increments of the continuous part of \(g\) on the two sides of the centre, and let \(A\) and \(B\) be the multisets of jump sizes to its right and to its left.  Then \(\mathcal L\) is Lebesgue measure on \([-c_-,c_+]\), plus an endpoint derivative functional governed by the elementary symmetric functions of \(A\), plus an absolutely continuous left tail whose density is the survival function of a sum of independent exponentials with means in \(B\); the positions of the jumps disappear.  Every quadruple \((c_-,c_+,A,B)\) occurs and is recovered from \(\mathcal L\), so the fibres are described exactly, and \(\mathcal L\) is representable by a positive measure precisely when \(|A|\le1\).  A weighted lattice extension, in which the transported moments are the factorial moments of the weight, is treated separately.
\end{abstract}

\maketitle

\section{Introduction}

Orthogonal polynomials are simultaneously algebraic objects, solutions of differential or difference equations, and encodings of moment functionals.  This interaction is particularly transparent for the classical continuous and discrete families; standard references include \cite{NSU,KOEKOEK,MR0481884}.  Perturbations of the underlying functional by point masses lead to Uvarov-type systems \cite{Uvarov1969,Koornwinder1984}, whereas perturbations by Dirac deltas and their derivatives lead naturally to quasi-definite, generally non-positive, linear functionals and have an established theory of their own \cite{AlvarezNodarse2004}.  Sobolev orthogonality provides another broad setting in which derivatives and several measures enter the scalar product \cite{Althammer,ISERLES1991151,survey}.

The present work originates in Stieltjes differential calculus.  A derivator \(g\) is a left-continuous non-decreasing function, its Lebesgue--Stieltjes measure is denoted by \(\mu_g\), and differentiation with respect to \(g\) accommodates an absolutely continuous evolution and jump contributions in a single calculus.  This framework goes back to Young and Daniell \cite{YOUNG,DANIELL}; modern formulations and applications can be found, for example, in \cite{RODRIGO,FRAN,FrTo,MaiaTojo}.  The recursively defined Stieltjes monomials \(g_{x_0,n}\) play the role of powers centred at \(x_0\) and satisfy
\[
 D_g g_{x_0,n}=n g_{x_0,n-1}.
\]
They therefore suggest a polynomial calculus adapted to \(D_g\).

There are two distinct orthogonality constructions.  The first uses the ordinary pointwise product and the measure \(\mu_g\).  For the model derivator considered below, this is exactly a classical Legendre functional with an interior mass, and hence belongs to the familiar Uvarov theory.  The second construction works inside the linear span \(\mathbb P_g\) of the Stieltjes monomials and introduces the product
\[
 g_{x_0,n}\star g_{x_0,m}=g_{x_0,n+m}.
\]
The decisive issue is whether this second construction creates a genuinely new abstract orthogonality or merely transports an ordinary polynomial orthogonality.

A previous version of this manuscript treated the \(\star\)-construction itself as the main source of novelty.  A referee correctly observed, in the one-jump model, that the correspondence \(t^n\mapsto g_{x_0,n}\) identifies the \(\star\)-algebra with the ordinary polynomial algebra and that the resulting family is a Uvarov--Legendre family in disguise.  Rather than defending the earlier interpretation, we take that observation as the starting point of the present paper and determine its full consequences.  This change of viewpoint leads to a substantially stronger question:
\begin{quote}
\emph{Which ordinary polynomial moment functionals can be produced by a Stieltjes derivator, and what information about the derivator survives?}
\end{quote}

The first main result is a transfer principle.  Under the natural linear-independence hypothesis on the Stieltjes monomials, \(\Phi(t^n)=g_{x_0,n}\) is an isomorphism of differential algebras and the \(\star\)-bilinear form is the pull-back of a single ordinary polynomial moment functional \(\mathcal L\).  Its moments satisfy the elementary but fundamental identity
\[
 m_n=\mathcal L(t^n)=
 \frac{g_{x_0,n+1}(b)-g_{x_0,n+1}(a)}{n+1}.
\]
Consequently, recurrence coefficients, orthogonality constants (squared norms in the positive-definite case) and differential equations of the \(\star\)-orthogonal family are transported ordinary polynomial data.  We state this explicitly because it precisely separates the part of the construction that is classical from the part that remains specific to Stieltjes calculus.

The second and central result is a complete image theorem for the unweighted correspondence \(g\mapsto\mathcal L\), with the centre constrained to the orthogonality interval.  The exponential generating function of the Stieltjes monomials is the Stieltjes exponential.  This converts the recursive construction into the closed formula
\[
 \mathcal L(e^{zt})=
 \frac{1}{z}\left(
 e^{c_+z}\Pi_A(z)-\frac{e^{-c_-z}}{\Pi_B(z)}
 \right),
 \qquad
 \Pi_A(z)=\prod_{\alpha\in A}(1+\alpha z),
 \quad
 \Pi_B(z)=\prod_{\beta\in B}(1+\beta z).
\]
Here \(c_-\) and \(c_+\) record only the total continuous increments on the two sides of the centre, while \(A\) and \(B\) record only jump sizes.  In particular, the positions of all jumps and the detailed shape of the continuous part disappear.  Read term by term, this identity says that \(\mathcal L\) is Lebesgue measure on \([-c_-,c_+]\), plus an endpoint derivative functional at \(c_+\) whose coefficients are the elementary symmetric functions of \(A\), plus an absolutely continuous tail on \((-\infty,-c_-]\) whose density is the survival function of a sum of independent exponential random variables with means \(\beta\in B\); for finitely many distinct means this is the classical hypoexponential survival function.  We prove not only this invariance but also the two converse statements that make the word ``image'' literal: every quadruple \((c_-,c_+,A,B)\) with \(c_\pm\ge0\) and summable multisets \(A,B\) of positive numbers is realized by an admissible derivator, and the meromorphic function \(z\mathcal L(e^{zt})\) determines \((c_-,c_+,A,B)\).  Thus the fibres are classified exactly.

The representation exposes an orientation effect.  Right-hand jumps enter through \(\Pi_A\) and hence generate endpoint derivative evaluations with coefficients equal to elementary symmetric functions.  Left-hand jumps enter through \(\Pi_B^{-1}\) and generate an absolutely continuous tail governed by the survival function of a summable sum of independent exponentials; for finitely many distinct rates this is the classical hypoexponential survival function.  The familiar one-jump Uvarov case is therefore only the first member of a hierarchy: two or more right jumps give derivative-of-Dirac perturbations, a setting already studied abstractly for general quasi-definite functionals \cite{AlvarezNodarse2004}, whereas the Stieltjes construction imposes the special symmetric-function coefficients and erases the jump locations.  On the left, the apparently infinite derivative expansion is better understood as an ordinary positive measure with an exponential, or more generally convolution-of-exponentials, tail.

A third main result gives a sharp positivity criterion.  The functional is representable by a positive Borel measure exactly when there is at most one right-hand jump, and it is positive definite precisely when, in addition, the representing measure has infinite support.  The proof is not based on finite Hankel experiments: it uses the boundedness of Fourier transforms of finite signed measures and the growth of \(\Pi_A(iy)\).  This also identifies a positive family of piecewise weights consisting of a flat interval joined to a left tail described by a convolution of exponentials.  We are not aware of a systematic treatment of this particular family.  There is, of course, extensive nearby literature on orthogonal polynomials for weights with jumps or discontinuities, including Gaussian examples \cite{XuZhao2011,Bogatskiy2016}; for this family we therefore formulate a precise open problem rather than an absolute novelty claim.

Finally, we distinguish the unweighted image theorem from weighted extensions.  For the unit-lattice derivator, weighting the \(\star\)-bilinear form turns ordinary moments into factorial moments.  Three classical weights then behave differently.  Meixner produces a scaled Gamma--Laguerre functional, Charlier degenerates to evaluation at one point, and Krawtchouk produces a finite-order derivative functional.  In particular, finite support of the original discrete weight must not be confused with positivity of the transported \(\star\)-functional.

The paper is organized as follows.  Section~\ref{sec:basic} recalls the Stieltjes calculus and the monomials.  Section~\ref{sec:gop} treats ordinary-product orthogonality, and Section~\ref{sec:gogp} defines the intrinsic \(\star\)-product.  Section~\ref{sec:transfer} proves the transfer principle.  Section~\ref{sec:image} contains the master generating identity, the image and fibre theorems, and the positivity classification.  Section~\ref{sec:dictionary} develops explicit right-jump, left-jump and weighted lattice realizations.  Section~\ref{sec:model} retains the one-jump model in full detail, including all explicit coefficients and its second-order \(g\)-differential equation, both as a worked realization and as an independent check of the structural theory.  The final section summarizes the conclusions and isolates the flat-plus-exponential-convolution family as an open problem.

\section{Basic definitions and notations}\label{sec:basic}

We begin this section with a summary of basic results related to Stieltjes derivatives, following \cite{FRAN,Cora,FRIGON}. As for classical univariate orthogonal polynomials, there are a number of {standard references}, such \cite{AAR,KOEKOEK,NSU}.

\begin{definition}
Let $g: {\mathbf{R}} \to {\mathbf{R}}$ be a left-continuous non-decreasing function. Under these {hypotheses} we shall refer to $g$ as a \emph{derivator}. In what follows any function {denoted by the letter} $g$ will refer to a function satisfying these two conditions.
\end{definition}

It is important to notice that derivators defined on a closed interval $[a,b] \subset {\mathbf{R}}$ can be extended to the whole real line.

\begin{definition}
Let $g$ be a derivator. Then, we can introduce a pseudometric
\begin{equation}
d_{g}(x,y)=\vert g(x) - g(y) \vert.
\end{equation}
\end{definition}

With this pseudometric we can introduce the \emph{open $g$-balls}:
\begin{definition}
Given any $x \in {\mathbf{R}}$ and $\epsilon >0$,
\begin{equation}
B_{g}(x,\epsilon):=\left \{ y \in {\mathbf{R}} \, \vert \, d_{g}(x,y) < \epsilon \right \}.
\end{equation}
\end{definition}

\begin{lemma}[{\cite{FRIGON}}]
We have that
\begin{equation}
\tau_{g}:=\left \{ U \subset {\mathbf{R}} \, \vert \, \forall x \in U \, \,\exists \epsilon>0 \, : \, B_{g}(x,\epsilon) \subset U \right \}
\end{equation}
is a topology where the {$g$-}balls are open sets.
\end{lemma}
{We will denote the usual topology by $\tau_{u}$.}

\begin{definition}
Given any {set} $X \subset {\mathbf{R}}$, we say a function $f:X \to {\mathbf{R}}$ is \emph{$g$-continuous at a point $x \in X$} if
\begin{equation}
\forall \epsilon >0 \,\, \exists \delta>0 \,\,\text{ such that {if}} \,\, y \in B_{g}(x,\delta)\cap X \implies \vert f(x)-f(y) \vert < \epsilon.
\end{equation}
We say $f$ is \emph{$g$-continuous on $X$} if it is $g$-continuous at every point $x \in X$.
\end{definition}

\begin{proposition}[{\cite[Proposition 3.2]{FRIGON}}]
If $f:X \to {\mathbf{R}}$ is $g$-continuous on $X$, then
\begin{enumerate}
\item $f$ is left-continuous at every point $x \in X$;
\item if $g$ is continuous at $x \in X$, then $f$ is continuous at $x \in X$;
\item if $g$ is constant on some interval {$[\alpha,\beta]\subset\mathbf{R}$}, then $f$ is {constant} on $[\alpha,\beta] \cap X$.
\end{enumerate}

\end{proposition}

\begin{definition}
{Denote by
\begin{equation}
g(x^{+})\equiv\lim_{y \to x^{+}} g(y) \in {\mathbf{R}},
\end{equation}
 the \emph{limit from the right of $g$ at the point $x$}} for any $x \in {\mathbf{R}}$. Let
\begin{equation}
\Delta g(x):=g(x^{+})-g(x).
\end{equation}
\end{definition}

\begin{definition}
Given a derivator $g$, let
\begin{equation}
D_{g}=\{ x \in {\mathbf{R}} \, \vert \, \Delta g(x) >0 \},
\end{equation}
the set of discontinuities of $g$. Let
\begin{equation}
C_{g}=\{ x \in {\mathbf{R}} \, \vert \, g \text{ is constant on } (x-\epsilon,x+\epsilon) \text{ for some } \epsilon>0 \}.
\end{equation}
\end{definition}

Note that $C_{g}$ is open in the usual topology. Hence, we can write
\begin{equation}
C_{g}=\bigcup_{n \in \Lambda} (a_{n},b_{n}),
\end{equation}
where $\Lambda \subset {\mathbf{N}}$, $(a_{k},b_{k}) \cap (a_{l},b_{l})=\emptyset$ if $k \neq l$ and $a_{n},b_{n} \in {\mathbf{R}} \cup \{+\infty,-\infty\}$ for $n \in \Lambda$. With these notations
\begin{equation}
N_{g}^{-}=\{a_{n} \}_{n \in \Lambda} \setminus D_{g}, \qquad N_{g}^{+}=\{b_{n} \}_{n \in \Lambda} \setminus D_{g}, \quad N_{g}=N_{g}^{-} \cup N_{g}^{+}.
\end{equation}

\begin{example}
Let us consider
\begin{equation}\label{eq:g0}
g_{0}(x)=x, \quad x \in [-1,1],
\end{equation}
which implies that $g_{0}$-derivative is the classical derivative. Then we have $C_{g_{0}}=D_{g_{0}}=N_{g_{0}}^{+}=N_{g_{0}}^{-}=\emptyset$ (i.e.{,} the empty set).
\end{example}

\begin{example}
If we consider a modification of $g_{0}$ defined in \eqref{eq:g0} by adding a jump discontinuity at $x_{0}=0$,
\begin{equation}\label{eq:ex1}
g_{1}(x;\alpha)=\begin{cases}
x, & x \in [-1,0], \\
x+\alpha{,} & x \in (0,1],
\end{cases} =\begin{cases}
x, & x \leq 0, \\
x+\alpha, & x>0{,}
\end{cases}
\end{equation}
then, if {$\alpha > 0$}, we have $C_{g_{1}}=N_{g_{1}}^{+}=N_{g_{1}}^{-}=\emptyset$, and $D_{g_{1}}=\{0\}$. As $\alpha \to 0${,} $g_{1}(x;\alpha)$ tends to $g_{0}(x)=x$.
\end{example}

\begin{definition}[{\cite[Definition 3.7]{FRAN}}]
Let $a,b \in {\mathbf{R}}$ be such that $a<b$, $a \not \in N_{g}^{-}$ and $b \not \in D_{g} \cup N_{g}^{+} \cup C_{g}$. We define the Stieltjes derivative or $g$-derivative of a function $f:[a,b] \to {\mathbf{R}}$ at a point $x \in [a,b]$ as
\begin{equation}
f'_{g}(x)=
\begin{cases}
\displaystyle \lim_{y \to x^{+}} \frac{f(y)-f(x)}{g(y)-g(x)}, & x \in D_{g}, \\[5mm]
\displaystyle \lim_{y \to x} \frac{f(y)-f(x)}{g(y)-g(x)}, & x \not \in D_{g} \cup C_{g}, \\[5mm]
\displaystyle \lim_{y \to b_{n}^{+}} \frac{f(y)-f(b_{n})}{g(y)-g(b_{n})}, & x \in (a_{n},b_{n}) \subset C_{g},
\end{cases}
\end{equation}
where $(a_{n},b_{n})$ is a connected component of $C_{g}$. {In the case $x \notin D_g \cup C_g$, it may happen that $x \in N_g$. In that situation, the limit must be understood in a one-sided sense, namely}
\begin{equation}
f_{g}'(x)=\begin{cases}
\displaystyle \lim_{y \to x^{+}} \frac{f(y)-f(x)}{g(y)-g(x)}, & x \in N_{g}^{+}, \\[5mm]
\displaystyle \lim_{y \to x^{-}} \frac{f(y)-f(x)}{g(y)-g(x)}, & x  {\in} N_{g}^{-}.
\end{cases}
\end{equation}
\end{definition}

For convenience, for each $t\in\mathbf{R}$, we consider the following notation from~\cite[Proposition~3.9]{FRAN}:
\begin{equation}\label{eq:notation of t^*}
t^*=\begin{cases}
b_n, & \text{if $t\in(a_n,b_n)\subset C_g$, and $b_n<\infty$,}\\
t, & \text{otherwise.}
\end{cases}
\end{equation}
Observe that $t^* \notin C_g$ for all $t\in \mathbf{R}$.

Let us consider the linear space ${\mathbb{P}}$ of polynomials
\begin{equation}
{\mathbb{P}}=\left\{ a_{0}+\sum_{i=1}^{n} a_{i,n} x^{i} \, \vert \, {a_0,}\ a_{i,n} \in {\mathbf{R}}, \quad n \in {\mathbf{N}} \right \}.
\end{equation}

{
\begin{definition}\label{def:mug}
Given a nondecreasing function $g$, we will denote the \emph{Lebesgue-Stieltjes measure generated by $g$} by $\mu_g$, that is, the Borel measure such that $\mu_g([a,b))=g(b)-g(a)$ for any $a,b\in\mathbf{R}$, $a\le b$. The integrals with respect to this measure will be denoted as
\[
\int_{A}f(s)\,d\mu_{g}(s)\equiv \int_{A} f\,d\mu_{g},
\]
for a $\mu_g$--measurable set $A$ and a $\mu_g$--measurable function $f:A\to\mathbf R$. We will use the following notation for integrals:
	
\[
\int_{x}^{y}f(s)\,d\mu_{g}(s):=\begin{dcases}
\int_{[x,y)} f(s)\, d\mu_g(s), &  y\geq x,\\
-\int_{[y,x)} f(s)\, d\mu_g(s), &  y< x.
\end{dcases}
\]
\end{definition}
}

We shall also use the linear space ${\mathbb{P}}_{g}$ of $g$-polynomials. In this regard, we need to introduce the $g$-monomials following \cite{Cora}.

\begin{definition}\label{def:gmonomials}
Let $g:{\mathbf{R}} \to {\mathbf{R}}$ be a derivator and let $x_{0} \in {\mathbf{R}}$. We define
\begin{equation}
g_{x_{0},0}(x)=1, \quad \forall x \in {\mathbf{R}},
\end{equation}
and for $n \in {\mathbf{N}}$ the $g$-monomials centered at $x_{0}\in\mathbf R$, where $x_{0}$ is called the center, $g_{x_{0},n}: {\mathbf{R}} \to {\mathbf{R}}$ are recursively defined by
{\begin{equation}\label{eq:gmonomials}
g_{x_{0},n}=
{n \int_{x_{0}}^x g_{x_{0},n-1} d\mu_{g}}, \quad \forall x \in {\mathbf{R}}.
\end{equation}}
\end{definition}

\begin{remark}
We would like to emphasize the following key property of $g$-monomials \cite{Cora}
\begin{equation}\label{eq:dgm}
(g_{x_{0},n})'_{g}(x)=n g_{x_{0},n-1}{(x)}.
\end{equation}
The idea of this relation is the same as in the continuous case for $x^{n}$ and of  \cite[Eq. (2.22)]{SERGEI} in the framework of orthogonal polynomials on nonuniform lattices.
\end{remark}

\begin{definition}[{\cite[Definition 3.1]{Cora}}]
We shall call $g$-polynomials centered at {$x_{0}$  the linear combinations} of $g$-monomials centered at $x_{0}$, and we shall consider the linear space of $g$-monomials
\begin{equation}
{\mathbb{P}}_{g}=\left\{ a_{0}+\sum_{i=1}^{n} a_{i,n} g_{x_{0},i}(x) \, \vert \, {a_{0},}\  a_{i,n} \in {\mathbf{R}}, \quad n \in {\mathbf{N}} \right \}.
\end{equation}
\end{definition}
Due to \cite[Proposition 3.14]{Cora}, this definition does not depend on the point $x_0$ chosen. {The work \cite{gmon} studies} the density of the $g$-monomials in appropriate space of functions.

\begin{definition}[{\cite{fernandez-tojo-2020}}]
We define $g^B:\mathbf{R}\to \mathbf{R}$ as:
\begin{displaymath}
g^B(t)=\begin{cases}
\displaystyle
\sum_{s\in [a,t)\cap D_g}\Delta  g(s), & t>a, \\
\displaystyle
-\sum_{s\in [t,a)\cap D_g} \Delta g(s), & t\leq a.
\end{cases}
\end{displaymath}
It is clear that $g^B$ is a left-continuous and non-decreasing function. Moreover, the map $g^C:\mathbf{R}\to\mathbf{R}$ given by
\begin{displaymath}
g^C(t):=g(t)-g^B(t),
\end{displaymath}
is also non-decreasing and continuous. We say {that} $g^C$  is the \emph{continuous part} of $g$ and {that} $g^B$ is the \emph{jump part} of $g$. {Similarly to Definition \ref{def:mug} it is possible to introduce the Lebesgue-Stieltjes measures $\mu_{g^{B}}$ and $\mu_{g^{C}}$ for the jump and continuous parts of the derivator $g$, respectively.}
\end{definition}

\begin{example}
Let us consider the derivator $g=g_{1}$ defined in \eqref{eq:ex1}. Next we shall compute the $g$-monomials for this derivator. For $n=0${, by convention, we set} $g_{0,0}(x)=1$. Moreover, for $n=1$ we have, first for $x \geq 0$
\begin{equation}
g_{0,1}(x)=\int_{[0,x)} 1 d\mu_{g}=g(x)-g(0)=g(x)=\begin{cases} x+\alpha, & x >0, \\ 0, & x=0. \end{cases}
\end{equation}
For $x <0$ we have,
\begin{equation}
g_{0,1}(x)=\int_{{[0,x)}} 1 d\mu_{g}=-\int_{[x,0)} 1 d\mu_{g}=-g(0)+g(x)=x.
\end{equation}
{Combining both parts yields}
\begin{equation}
g_{0,1}(x)=\begin{cases}  x, & x \leq 0, \\ x+\alpha, & x>0. \end{cases}
\end{equation}
Similarly, for $n=2$ we have, first for $x \geq 0$,
\begin{equation}
g_{0,2}(x)=2 \int_{[0,x)} g_{0,1}(s) d\mu_{g}(s)=2 \int_{[0,x)} g_{0,1}(s) d\mu_{g^{B}}(s)+2 \int_{[0,x)} g_{0,1}(s) d\mu_{g^{C}}(s).
\end{equation}
The first term can be easily computed since the derivator $g$ has only one jump point
\begin{equation}
2 \int_{[0,x)} g_{0,1}(s) d\mu_{g^{B}}(s)=2 g_{0,1}(0)\mu_{g}(\{0\})=0.
\end{equation}
The second term is also easy to compute since we can use Riemann integrals to obtain
\begin{equation}
2 \int_{[0,x)} g_{0,1}(s) d\mu_{g^{C}}(s)=2 \int_{0}^{x} s+\alpha ds=x^{2}+2\alpha x.
\end{equation}
For $x<0$ since we do not have the jump point, we can use directly Riemann integration to get
\begin{equation}
2 \int_{[0,x)} g_{0,1}(s) d\mu_{g^{C}}(s)=-2 \int_{[x,0)} g_{0,1}(s) d\mu_{g}(s)=x^{2}.
\end{equation}
Again, we can summarize this second monomial as
\begin{equation}
g_{0,2}(x)=\begin{cases} x^{2}, & x \leq 0, \\ x^{2}+2\alpha x, & x>0. \end{cases}
\end{equation}
If we proceed by induction it is possible to obtain the general form for the $g$-monomials
\begin{equation}\label{eq:explicitmonomials}
g_{0,n}(x)=\begin{cases} x^{n}, & x \leq 0, \\ x^{n}+n \alpha x^{n-1}, & x>0. \end{cases}
\end{equation}
If we consider the limit as $\alpha \to 0$ in the above expression we recover the classical monomials $x^{n}$.

\end{example}

\begin{remark}\label{remark:1}
Let us consider two $g$-monomials of degree $p$ and $q$, respectively. For $x>0$ we have
\begin{equation}
g_{p}(x)=x^{p}+p \alpha x^{p-1}, \quad g_{q}(x)=x^{q}+q \alpha x^{q-1}.
\end{equation}
If we compute
\begin{multline}
g_{p}(x)g_{q}(x)=\left(x^{p}+p \alpha x^{p-1}\right)\left( x^{q}+q \alpha x^{q-1} \right) \\
=\alpha ^2 p q x^{p+q-2}+\alpha  p x^{p+q-1}+\alpha  q x^{p+q-1}+x^{p+q}=g_{p+q}(x)+\alpha^{2} p q x^{p+q-2},
\end{multline}
i.e.
\begin{equation}
g_{p}(x)g_{q}(x) \neq g_{p+q}(x),
\end{equation}
but
\begin{equation}\label{eq:promo}
g_{p}(x)g_{q}(x)=g_{p+q}(x)+\alpha^{2} p q x^{p+q-2}.
\end{equation}
\end{remark}

\section{$g$-orthogonal polynomials}\label{sec:gop}

\begin{definition}
Let $g$ be a derivator. Let $\{P_{n}(x)\}_{n}$ be a sequence of polynomials, where {$\text{deg}(P_{n})=n$}. We shall say that $\{P_{n}(x)\}_{n}$ is \emph{$g$-orthonormal polynomial {sequence}} on {the interval} $[a,b]$ if
\begin{equation}
\langle P_{n},P_{m} \rangle_{g}:=\int_{[a,b)} P_{n}(s)P_{m}(s)d\mu_{g}(s)=\delta_{n,m}=\begin{cases} 0, & n \neq m, \\ 1, & n=m.  \end{cases}
\end{equation}
\end{definition}

{\begin{remark}
For the derivator $g_1$ defined in \eqref{eq:ex1}, the induced inner product
\[
\langle u,v\rangle_g=\int_{-1}^{1}u(s)v(s)\,ds+\alpha\,u(0)v(0)
\]
coincides with the classical Legendre weight perturbed by an interior Dirac mass at $0$.
This corresponds to a well-known Uvarov-type modification of the measure and therefore the orthogonality structure in this section is not new in itself.
Our contribution here is to reinterpret this classical perturbation within the framework of Stieltjes derivatives; this serves as the point of comparison for the $\star$-based construction of Section~\ref{sec:gogp} and for the transfer principle of Section~\ref{sec:transfer}.
\end{remark}
}

\begin{example}
We shall first compute the first terms of the sequence of $g$-Legendre orthonormal polynomials on $[-1,1]$ by considering the derivator $g$ defined in \eqref{eq:ex1}.

First of all, we compute
\begin{multline}
\langle 1, 1 \rangle_{g}=\int_{[-1,1)}  d\mu_{g}(s)=
\int_{[-1,0)} 1 d\mu_{g}(s)+\int_{[0,1)} 1 d\mu_{g}(s) \\ =
\int_{[-1,0)} 1 {d\mu_g(s)}+\int_{[0,1)} 1 d\mu_{g^{B}}(s)+\int_{[0,1)} 1 d\mu_{g^{C}}(s)={2+\alpha}.
\end{multline}
Therefore, the orthonormal $g$-Legendre polynomial of degree 0 is given by
\begin{equation}
P_{0}^{g}(x)=\frac{1}{\sqrt{2+\alpha}}.
\end{equation}

Let us now consider a $g$-polynomial of degree $1$,
\begin{equation}
P_{1}^{g}(x)=\nu_{0,1} +\nu_{1,1} x.
\end{equation}
If we impose the conditions
\begin{equation}
\langle P_{1}^{g},P_{0}^{g}\rangle_{g}=\int_{[-1,1)} P_{1}^{g}(s)P_{0}^{g}(s) d\mu_{g}(s)=0, \quad \langle P_{1}^{g},P_{1}^{g}\rangle_{g}=\int_{[-1,1)} P_{1}^{g}(s)P_{1}^{g}(s) d\mu_{g}(s)=1,
\end{equation}
we obtain
\begin{equation}
\nu_{0,1}=0, \quad \nu_{1,1}=\sqrt{\frac{3}{2}}.
\end{equation}
Therefore, the $g$-Legendre orthonormal polynomial of degree $1$ with respect to the derivator $g$ defined in \eqref{eq:ex1} is given by
\begin{equation}
P_{1}^{g}(x)=\sqrt{\frac{3}{2}} x,
\end{equation}
which coincides with the classical orthonormal Legendre polynomial.

Similarly,
\begin{equation}
P_{2}^{g}(x)=\frac{\sqrt{\frac{5}{2}} \left(3 (\alpha +2) x^2-2\right)}{\sqrt{(\alpha +2) (9 \alpha   +8)}}.
\end{equation}
If we compute the limit as $\alpha \to 0$ we obtain
\begin{equation}
\lim_{\alpha \to 0} P_{2}^{g}(x)=\frac{1}{2} \sqrt{\frac{5}{2}} \left(3 x^2-1\right),
\end{equation}
i.e. the Legendre orthonormal polynomial of degree $2$.
\end{example}

\subsection{Inner product approach}

Let $g$ be a derivator and let us consider
\begin{equation}
\langle u,v \rangle_{g}=\int_{[a,b)} u \, v\, d\mu_{g}.
\end{equation}
Let us introduce the moments associated with this inner product for a basis $\{\theta_{n}(x)\}$ of the space of polynomials of degree at most $n$:
\begin{equation}
\mu_{i,j}=\langle \theta_{i}(x) ,\theta_{j}(x)  \rangle_{g} = \int_{[a,b)} \theta_{i}(x) \theta_{j}(x) d\mu_{g}.
\end{equation}
If $\theta_{n}(x)$ is a monic polynomial {of degree $n$}, then with these notations we can introduce the monic $g$-orthogonal polynomials in determinant form by using the moments:
\begin{equation}\label{eq:polmom}
\hat{P}_{n}^{g}(x)=\frac{1}{\Delta_{n}} \begin{vmatrix}
\mu_{0,0} & \mu_{0,1} & \cdots & \mu_{0,n} \\
\mu_{1,0} & \mu_{1,1} & \cdots & \mu_{1,n} \\
\vdots & \vdots &             &\vdots \\
\mu_{n-1,0} & \mu_{n-1,1} & \cdots &\mu_{n-1,n} \\
\theta_{0} & \theta_{1} & \cdots & \theta_{n}
\end{vmatrix},
\end{equation}
where, as usual, \cite{MR0481884}
\begin{equation}
\Delta_{n}=\begin{vmatrix}
\mu_{0,0} & \mu_{0,1} & \cdots & \mu_{0,n-1} \\
\mu_{1,0} & \mu_{1,1} & \cdots & \mu_{1,n-1} \\
\vdots & \vdots &             &\vdots \\
\mu_{n-1,0} & \mu_{n-1,1} & \cdots &\mu_{n-1,n-1}
\end{vmatrix}.
\end{equation}

{The determinantal construction in \eqref{eq:polmom} requires that the corresponding Hankel determinants $\Delta_n$ do not vanish, i.e., that the moment functional is quasi-definite. For the derivator $g_1(x;\alpha)$ with $\alpha>0$ introduced in \eqref{eq:ex1}, this property follows from the positive definiteness of the associated moment functional, as shown by the explicit Hankel determinant formula below.}

\begin{example}
For the derivator $g=g_{1}$ defined in \eqref{eq:ex1} we can prove the following Lemma.
\begin{lemma}
Let $g=g_{1}$ be the derivator defined in \eqref{eq:ex1}, and let $\theta_{n}=x^{n}$. Then, the moments can be explicitly computed as
\begin{equation}
\mu_{0,0}=2+\alpha, \quad \mu_{i,j}=\frac{(-1)^{i+j}+1}{i+j+1}, \quad i+j>0.
\end{equation}	
\end{lemma}
\begin{proof}
First of all, we compute
\begin{equation}\label{eq:mom0}
\mu_{0,0}=\langle 1,1 \rangle_{g} = \int_{[-1,1)} 1  d\mu_{g}=\int_{-1}^{1} ds+ \alpha=2+\alpha.
\end{equation}
In general for $i+j> 0$
\begin{equation}\label{eq:momij}
\mu_{i,j}=\langle x^{i},x^{j} \rangle_{g} = \int_{[-1,1)} s^{i+j} d\mu_{g}
=\int_{-1}^{1} s^{i+j} ds+ \left. x^{i+j} \right \vert_{x=0} \alpha =\frac{(-1)^{i+j}+1}{i+j+1}.\qedhere
\end{equation}
\end{proof}

{
\begin{remark}
By using the explicit form of the moments \eqref{eq:mom0} and \eqref{eq:momij}, the Hankel determinants
\[
\Delta_n=\det\left(\mu_{i,j}\right)_{i,j=0}^{n-1}
\]
can be obtained as a rank-one perturbation of the classical Legendre Hankel
determinants. Indeed, let
\begin{equation}\label{eq:hankelc}
H_n=
\left(
\int_{-1}^{1}x^{i+j}\,dx
\right)_{i,j=0}^{n-1}.
\end{equation}
Since the only modified moment is \(\mu_{0,0}\), we have
\[
\left(\mu_{i,j}\right)_{i,j=0}^{n-1}
=
H_n+\alpha e_0e_0^T,
\]
where \(e_0=(1,0,\ldots,0)^T\). Therefore, by the matrix determinant lemma,
\[
\Delta_n
=
\det H_n
\left(
1+\alpha e_0^TH_n^{-1}e_0
\right).
\]
The quantity \(e_0^TH_n^{-1}e_0\) is the Christoffel-Darboux kernel associated
with the monic Legendre polynomials evaluated at \(0\), namely
\[
e_0^TH_n^{-1}e_0=K_{n-1}(0,0).
\]
Hence
\[
\Delta_n
=
\det H_n
\left(
1+\alpha K_{n-1}(0,0)
\right).
\]
Using
\begin{equation}\label{eq:dethankelc}
\det H_n
=
\prod_{k=0}^{n-1}
\frac{2^{2k+1}}{(2k+1)\binom{2k}{k}^2}
\end{equation}
and
\[
K_{n-1}(0,0)
=
2^{1-2n}n^2
\binom{n-1}{\left\lfloor \frac{n-1}{2}\right\rfloor}^{2},
\]
we obtain
\[
\Delta_n
=
\left(
\prod_{k=0}^{n-1}
\frac{2^{2k+1}}{(2k+1)\binom{2k}{k}^2}
\right)
\left(
1+
\alpha\,2^{1-2n}n^2
\binom{n-1}{\left\lfloor \frac{n-1}{2}\right\rfloor}^{2}
\right).
\]
Equivalently,
\[
\Delta_n
=
2^{(n-1)^2}
\left(
2^{2n-1}
+
\alpha n^2
\binom{n-1}{\left\lfloor \frac{n-1}{2}\right\rfloor}^{2}
\right)
\prod_{k=0}^{n-1}
\frac{(k!)^4}{(2k+1)((2k)!)^2}.
\]
In particular,
\[
\Delta_n>0,
\qquad n\geq1,
\]
for every \(\alpha>0\).
\end{remark}
}
	
It is easy to check that the moments $\mu_{i,j}$ converge to the classical moments of Legendre as $\alpha \to 0$. In fact, for this specific example with $g$ defined in \eqref{eq:ex1} the only difference appears at the moment $\mu_{0,0}$.
	
Therefore, the monic $g$-Legendre orthogonal polynomials can be introduced in determinant form by using the moments
\begin{equation}
\hat{P}_{n}^{g}(x)=\frac{1}{\Delta_{n}^{g}} \begin{vmatrix}
\mu_{0,0} & \mu_{0,1} & \cdots & \mu_{0,n} \\
\mu_{1,0} & \mu_{1,1} & \cdots & \mu_{1,n} \\
\vdots & \vdots &             &\vdots \\
\mu_{n-1,0} & \mu_{n-1,1} & \cdots &\mu_{n-1,n} \\
1 & x & \cdots & x^{n}
\end{vmatrix}.
\end{equation}

We obtain the very same orthogonal polynomials as computed before, up to {an} appropriate normalization. For instance
\begin{align}
\hat{P}_{1}^{g}(x)&= x=\sqrt{\frac{2}{3}}P_{1}^{g}(x), \\
\hat{P}_{2}^{g}(x)&=x^2-\frac{2}{3 (\alpha +2)}=\frac{\sqrt{\frac{2}{5}} \sqrt{(\alpha +2) (9 \alpha +8)}}{3 (\alpha +2)} P_{2}^{g}(x), \\
\hat{P}_{3}^{g}(x)&= x^3-\frac{3 x}{5}= \frac{2 \sqrt{\frac{2}{7}}}{5} P_{3}^{g}(x).
\end{align}
\end{example}

\subsection{Recurrence relation}

If $\mu$ is a measure supported on the real line such that for the corresponding inner product $\langle x f(x),g(x) \rangle = \langle f(x),xg(x) \rangle$, then the sequence of orthonormal polynomials $\{p_{n}(x)\}_{n}$ satisfy a three term recurrence relation
\begin{equation}
x p_{n}(x)=a_{n+1} p_{n+1}(x)+b_{n} p_{n}(x)+a_{n}p_{n-1}(x), \quad n \geq 0 ,
\end{equation}
with
\begin{equation}
p_{-1}(x):=0, \quad p_{0}(x)=\frac{1}{\left(\int_{\supp \mu} d\mu(s) \right)^{1/2}},
\end{equation}
where
\begin{equation}
a_{n}=\int_{\supp \mu} x p_{n-1}(x)p_{n}(x)d\mu(x), \quad b_{n}=\int_{\supp \mu} x p_{n}^{2}(x) d\mu(x).
\end{equation}
For inner products that  $\langle x f(x),g(x) \rangle \neq \langle f(x),xg(x) \rangle$  this type of recurrence relation is {no} longer true (see{, e.g.,} \cite{MARCELLAN1996245}).

\begin{example}
Let $g$ be the derivator defined in \eqref{eq:ex1}. Then,
\begin{equation}
\langle u,v \rangle_{g}=\int_{[-1,1)} u \, v\, d\mu_{g}=\int_{-1}^{1} u(s)v(s)ds+\alpha u(0) v(0).
\end{equation}
Let $\{\hat{P}_{n}^{g}(x)\}_{n}$ be the sequence of monic orthogonal polynomials with respect to the above inner product, and let $\{\hat{P}_{n}(x)\}_{n}$ be the sequence of classical monic Legendre polynomials. It is easy to prove that the sequence $\{\hat{P}_{n}^{g}(x)\}_{n}$ satisfies a three-term relation of the form
\begin{equation}\label{eq:ttrrp}
x  \hat{P}_{n}^{g}(x)=  \hat{P}_{n+1}^{g}(x) + u_{n}  \hat{P}_{n-1}^{g}(x).
\end{equation}
The result follows from the expansion
\begin{equation}
x  \hat{P}_{n}^{g}(x) = \hat{P}_{n+1}^{g}(x)  + \sum_{j=0}^{n} u_{j,n} \hat{P}_{j}^{g}(x).
\end{equation}
The Fourier coefficients are
\begin{equation}
u_{j,n}=\frac{\langle x \hat{P}_{n}^{g}(x),\hat{P}_{j}^{g}(x) \rangle_{g}}{\langle \hat{P}_{j}^{g}(x) ,\hat{P}_{j}^{g}(x) \rangle_{g}}.
\end{equation}
In this very specific situation we have that
\begin{equation}
\langle x \hat{P}_{n}^{g}(x),\hat{P}_{j}^{g}(x) \rangle_{g} = \langle  \hat{P}_{n}^{g}(x),x \hat{P}_{j}^{g}(x) \rangle_{g},
\end{equation}
and the result follows from the orthogonality and symmetry of the inner product. It is important to notice that in general the latter property is {no} longer true.

Following \cite{manolo} we can derive some further properties in this case. If we expand the monic $g$-Legendre polynomials in terms of classical monic Legendre polynomials
\begin{equation}
\hat{P}_{n}^{g}(x)=\hat{P}_{n}(x)+\sum_{j=0}^{n-1} C_{j,n} \hat{P}_{j}(x),
\end{equation}
then
\begin{equation}
C_{j,n}=\frac{\int_{-1}^{1} \hat{P}_{n}^{g}(x)  \hat{P}_{j}(x) dx}{\int_{-1}^{1}  \hat{P}_{j}(x)  \hat{P}_{j}(x) dx}=-\frac{\alpha \hat{P}_{n}^{g}(0) \hat{P}_{j}(0)}{ \|  \hat{P}_{j}(x) \|^{2}}, \quad 0 \leq j \leq n-1,
\end{equation}
by using the orthogonality of $Q_{n}$ with respect to $P_{j}$, where the explicit value of the norm
\begin{equation}\label{eq:norm}
\|  \hat{P}_{j}(x) \|^{2}=\frac{2^{2 j+1}}{(2 j+1) \binom{2 j}{j}^2}.
\end{equation}
If we introduce the kernels associated to the Legendre polynomials
\begin{equation}
K_{n}(x,y)=\sum_{j=0}^{n} \frac{\hat{P}_{j}(x) \hat{P}_{j}(y)}{\|  \hat{P}_{j}(x) \|^{2}}
\end{equation}
then we have
\begin{equation}\label{eq:auxpi}
\hat{P}_{n}^{g}(x)=\hat{P}_{n}(x)-\alpha \hat{P}_{n}^{g}(0) K_{n-1}(x,0).
\end{equation}
In particular
\begin{equation}\label{eq:aux3}
\hat{P}_{n}^{g}(0)=\hat{P}_{n}(0)-\alpha \hat{P}_{n}^{g}(0) K_{n-1}(0,0),
\end{equation}
where we can compute explicitly
\begin{equation}\label{eq:kernel0}
K_{n}(0,0)=2^{-2 n-1} (n+1)^2 \binom{n}{\left\lfloor \frac{n}{2}\right\rfloor }^2 >0.
\end{equation}
From \eqref{eq:aux3}, by using
\begin{equation}\label{eq:leg0}
\hat{P}_{n}(0)=\frac{\sqrt{\pi } \Gamma \left(\frac{n+1}{2}\right)}{\Gamma \left(\frac{1}{2}-\frac{n}{2}\right) \Gamma \left(n+\frac{1}{2}\right)},
\end{equation}
we obtain the explicit value
\begin{equation}\label{eq:aux4}
\hat{P}_{n}^{g}(0)=\frac{\hat{P}_{n}(0)}{1+\alpha K_{n-1}(0,0)}=\frac{\sqrt{\pi } 8^n}{\binom{2 n}{n} \Gamma \left(\frac{1}{2}-\frac{n}{2}\right) \Gamma \left(\frac{n}{2}+1\right) \left(2 \alpha  n^2 \binom{n-1}{\left\lfloor \frac{n-1}{2}\right\rfloor
   }^2+4^n\right)},
\end{equation}
so that \eqref{eq:auxpi} can be written as
\begin{equation}\label{eq:auxf}
\hat{P}_{n}^{g}(x)=\hat{P}_{n}(x)-\alpha \frac{\hat{P}_{n}(0)}{1+\alpha K_{n-1}(0,0)} K_{n-1}(x,0).
\end{equation}

Finally, we shall compute explicitly the coefficients $u_{n}$ in {the} three-term recurrence relation \eqref{eq:ttrrp} for the monic $g$-Legendre polynomials.
In doing so, we substitute three times \eqref{eq:auxf} into \eqref{eq:ttrrp} in order to obtain
\begin{multline}\label{eq:auxss}
x \left( \hat{P}_{n}(x)-\alpha \frac{\hat{P}_{n}(0)}{1+\alpha K_{n-1}(0,0)} K_{n-1}(x,0) \right) \\ =\left[ \hat{P}_{n+1}(x)-\alpha \frac{\hat{P}_{n+1}(0)}{1+\alpha K_{n}(0,0)} K_{n}(x,0) \right] \\
+u_{n} \left \{ \hat{P}_{n-1}(x)-\alpha \frac{\hat{P}_{n-1}(0)}{1+\alpha K_{n-2}(0,0)} K_{n-2}(x,0) \right \}.
\end{multline}
Now, we use the three-term recurrence relation for monic classical Legendre polynomials
\begin{equation}
x \hat{P}_{n}(x)=\hat{P}_{n+1}(x)+\frac{n^{2}}{(2n-1)(2n+1)} \hat{P}_{n-1}(x)
\end{equation}
and the explicit expression of the kernels, so that the left-hand side reads
\begin{multline}\label{eq:auxkk}
x \left( \hat{P}_{n}(x)-\alpha \frac{\hat{P}_{n}(0)}{1+\alpha K_{n-1}(0,0)} K_{n-1}(x,0) \right) \\
=\hat{P}_{n+1}(x)+\frac{n^{2}}{(2n-1)(2n+1)} \hat{P}_{n-1}(x) \\
- \alpha \frac{\hat{P}_{n}(0)}{1+\alpha K_{n-1}(0,0)} \sum_{j=0}^{n-1} \frac{ \hat{P}_{j+1}(x)+\frac{j^{2}}{(2j-1)(2j+1)} \hat{P}_{j-1}(x)  }{\| \hat{P}_{j}(x) \|^{2}} \hat{P}_{j}(0).
\end{multline}
Since
\begin{equation}
\hat{P}_{n}(x)=x^{n}+\frac{(n-1) n}{2(1-2 n)} x^{n-2}+\cdots,
\end{equation}
by equating the coefficient in $x^{n-1}$ in \eqref{eq:auxss} and taking into account \eqref{eq:auxkk}, we obtain the explicit expression for the coefficients $u_{n}$ in the three-term recurrence relation \eqref{eq:ttrrp}
\begin{equation}
u_{n}=\frac{n^{2}}{4n^{2}-1}+\alpha \left( \frac{ \hat{P}_{n-1}(0) \hat{P}_{n+1}(0)}{(1+\alpha K_{n}(0,0)) \| \hat{P}_{n-1}(x) \|^{2}} -\frac{ \hat{P}_{n-2}(0) \hat{P}_{n}(0)}{(1+\alpha K_{n-1}(0,0)) \| \hat{P}_{n-2}(x) \|^{2}}\right).
\end{equation}
Since all the values in the latter expression have been explicitly presented in \eqref{eq:norm}, \eqref{eq:kernel0}, and \eqref{eq:leg0}, we can also write
\begin{multline}
u_{n}=-\frac{\pi  \alpha  4^n n (2 n-3) \binom{2 (n-2)}{n-2}}{\binom{2 n}{n} \Gamma \left(\frac{1}{2}-\frac{n}{2}\right) \Gamma \left(\frac{3}{2}-\frac{n}{2}\right) \Gamma \left(\frac{n}{2}+1\right)^2 \left(2 \alpha  n^2 \binom{n-1}{\left\lfloor \frac{n-1}{2}\right\rfloor }^2+4^n\right)}+ \\
n \left(\frac{n}{4 n^2-1}-\frac{\sqrt{\pi } \alpha  16^n \Gamma \left(n+\frac{1}{2}\right)}{\binom{2 (n+1)}{n+1} \Gamma \left(1-\frac{n}{2}\right)^2 \Gamma (n) \Gamma \left(\frac{n+1}{2}\right)    \Gamma \left(\frac{n+3}{2}\right) \left(\alpha  (n+1)^2 \binom{n}{\left\lfloor \frac{n}{2}\right\rfloor }^2+2^{2 n+1}\right)}\right).
\end{multline}
As a consequence, we can generate the monic orthogonal polynomials $\hat{P}_{n}^{g}(x)$, giving rise e.g. to
\begin{align}
&\hat{P}_{-1}(x)=0, \quad \hat{P}_{0}(x)=1, \quad \hat{P}_{1}(x)=x, \quad \hat{P}_{2}(x)=x^2-\frac{2}{3 (\alpha +2)}, \\
&\hat{P}_{3}(x)=x^3-\frac{3 x}{5}, \\
&\hat{P}_{4}(x)=\frac{35 (9 \alpha +8) x^4-15 (15 \alpha +16) x^2+24}{35 (9 \alpha +8)}, \\
&\hat{P}_{5}(x)=x^5-\frac{10 x^3}{9}+\frac{5 x}{21}, \\
&\hat{P}_{6}(x)=\frac{21 x^2 \left(11 (225 \alpha +128) x^4-30 (105 \alpha +64) x^2+875 \alpha +640\right)-640}{231 (225 \alpha +128)},
\end{align}
where we observe a simple fact from the definition of the inner product: the odd polynomials coincide with classical monic Legendre polynomials.

In Figure \ref{figure1} we show the first five monic orthogonal $g$-Legendre polynomials $\hat{P}_{n}^{g}(x)$ for $\alpha=1$.

\begin{figure}[ht!]
\centering
\includegraphics[scale=0.9]{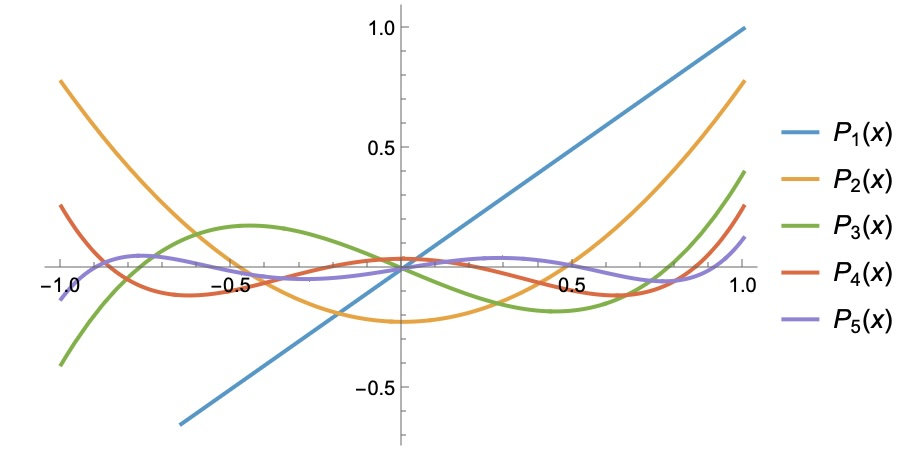}
\caption{The first five monic orthogonal $g$-Legendre polynomials $\hat{P}_{n}^{g}(x)$ for $\alpha=1$. It is possible to observe the interlacing property of the zeros.} \label{figure1}
\end{figure}

\end{example}

\subsection{Second-order differential equation}

According to \cite[Section 8]{survey}, solutions of the second-order differential equation
\begin{equation}
(a_{0}+a_{1}x+a_{2}x^{2}) y''(x)+(b_{0}+b_{1}x)y'(x)+\lambda y(x)=0, \qquad a_{i},b_{i},\lambda \in {\mathbf{R}}
\end{equation}
do not lead to new families of Sobolev orthogonal polynomials. In our case, since polynomials of odd degree do coincide with classical Legendre polynomials, we have that
\begin{equation}
(1-x^{2}) {\frac{d^{2}}{dx^{2}}} P_{2n+1}^{g}(x) -2x \frac{d}{dx} P_{2n+1}^{g}(x)+(2n+1)(2n+2) P_{2n+1}^{g}(x)=0, \quad n=0,1,\dots.
\end{equation}
For polynomials of even degree, we can obtain a differential equation following \cite{manolo}. In this case, we obtain a second order differential equation, but the polynomials are of degrees 5, 4, and 3, respectively, and they depend on $n$. More precisely
\begin{equation}
\sigma_{5}(x;n) {\frac{d^{2}}{dx^{2}}} P_{2n}^{g}(x) +\tau_{4}(x;n) \frac{d}{dx} P_{2n}^{g}(x)+\lambda_{3}(x;n) P_{2n}^{g}(x)=0, \quad n=1,2,\dots,
\end{equation}
where
\begin{multline}
\sigma_{5}(x;n)=16 x \left(x^2-1\right) \left(-\alpha ^2 2^{24 n} n^4 \Gamma (n)^2 \Gamma \left(2 n+\frac{1}{2}\right)^4 \right. \\
\left. -x^2 \binom{4 n}{2 n}^4 \Gamma \left(\frac{1}{2}-n\right)^4 \Gamma (n+1)^6
   \Gamma (2 n+1)^4 \left(8 \alpha  n^2 \binom{2 n-1}{\left\lfloor n-\frac{1}{2}\right\rfloor }^2+16^n\right)^2 \right. \\
   \left. +\pi  \alpha  \left(-2^{4 n+3}\right) (1-4 n)^2 n^6 \binom{4 n}{2 n}^2
   \binom{4 n-2}{2 n-1}^2 \left((4 n+1) x^2-1\right) \Gamma \left(\frac{1}{2}-n\right)^2 \Gamma (n)^4 \right. \\
   \left.  \Gamma (2 n)^4 \left(8 \alpha  n^2 \binom{2 n-1}{\left\lfloor    n-\frac{1}{2}\right\rfloor }^2+16^n\right)\right),
\end{multline}
\begin{multline}
\tau_{4}(x;n)=-\alpha ^2 2^{24 n+5} n^4 \left(2 x^2-1\right) \Gamma (n)^2 \Gamma \left(2 n+\frac{1}{2}\right)^4 \\
-16 \binom{4 n}{2 n}^2 \Gamma \left(\frac{1}{2}-n\right)^2 \Gamma (n+1)^4 \left(8   \alpha  n^2 \binom{2 n-1}{\left\lfloor n-\frac{1}{2}\right\rfloor }^2+16^n\right) \\
\times \left(2 x^4 \binom{4 n}{2 n}^2 \Gamma \left(\frac{1}{2}-n\right)^2 \Gamma (n+1)^2 \Gamma (2 n+1)^4
   \left(8 \alpha  n^2 \binom{2 n-1}{\left\lfloor n-\frac{1}{2}\right\rfloor }^2+16^n\right) \right. \\
   \left. +\pi  \alpha  16^n \left((4 n+1) x^4-2 x^2+1\right) \Gamma (4 n+1)^2\right),
\end{multline}
and
\begin{multline}
\lambda_{3}(x;n)=32 n (2 n+1) x \left(\alpha ^2 2^{24 n} (n-1) n \Gamma (n+1)^2 \Gamma \left(2 n+\frac{1}{2}\right)^4 \right. \\
\left. +x^2 \binom{4 n}{2 n}^4 \Gamma \left(\frac{1}{2}-n\right)^4 \Gamma (n+1)^6 \Gamma
   (2 n+1)^4 \left(8 \alpha  n^2 \binom{2 n-1}{\left\lfloor n-\frac{1}{2}\right\rfloor }^2+16^n\right)^2 \right. \\
   \left. +\pi  \alpha  2^{4 n+3} (1-4 n)^2 n^6 \binom{4 n}{2 n}^2 \binom{4 n-2}{2 n-1}^2
   \left((4 n+1) x^2-3\right) \Gamma \left(\frac{1}{2}-n\right)^2 \Gamma (n)^4 \Gamma (2 n)^4  \right. \\
   \left. \left(8 \alpha  n^2 \binom{2 n-1}{\left\lfloor n-\frac{1}{2}\right\rfloor}^2+16^n\right)\right).
\end{multline}

\section{$g$-orthogonal $g$-polynomials}\label{sec:gogp}

In contrast with Section~\ref{sec:gop} (where we worked with ordinary polynomials and the classical product) our objective in this section is to set up an intrinsic orthogonality on the space of $g$-polynomials themselves, i.e. on linear combinations of the $g$-monomials.

A fundamental obstacle is that the usual pointwise product does not preserve the space $\mathbb P_g$.  As already noted in Remark~\ref{remark:1}, the product of two $g$-monomials is not, in general, a single $g$-monomial:
\[
 g_{x_{0},n}(x)g_{x_{0},m}(x)\neq g_{x_{0},n+m}(x).
\]
More strongly, the pointwise product of two $g$-monomials need not be $g$-continuously $g$-differentiable.  By contrast, each individual $g$-monomial is infinitely $g$-differentiable; see~\cite{Cora}.  Moreover, let
\begin{equation}
P_{n}(x;g)=c_{n,n} g_{x_{0},n}(x)+\sum_{j=0}^{n-1} c_{n,j} g_{x_{0},j}(x), \quad n=0,1,2,\dots,
\end{equation}
be a sequence of $g$-polynomials. If we try to expand
\begin{equation}
g_{x_{0},1}(x)P_{n}(x;g)=\sum_{j=0}^{n+1} d_{n,j} g_{x_{0},j}(x),
\end{equation}
we observe that it is impossible to find the coefficients $d_{n,j}$ even for a simple case as \eqref{eq:ex1}, that is, ${\mathbb P}_{g}$ is not closed under the usual product of functions. Therefore, we need to introduce a new operation in the set of $g$-polynomials ${\mathbb P}_{g}$.

For this reason we introduce a product $\star$, under which the class of $g$-polynomials is stable, and we develop the corresponding orthogonality.  A coefficientwise product on functions is meaningful only when the chosen $g$-monomial expansion is unique.  Accordingly, from this point onward we restrict to pairs $(g,x_0)$ for which $\{g_{x_0,n}\}_{n\ge0}$ is linearly independent.  This hypothesis is formalized as \emph{admissibility} in Definition~\ref{def:admissible}; Proposition~\ref{prop:admissible} gives a simple sufficient criterion, and every derivator used in the explicit constructions below is verified to be admissible.  Under this standing restriction, each element of $\mathbb P_g$ has a unique finite coefficient expansion, so the following definition is unambiguous.

The remaining sections first identify the exact algebraic status of this product (Section~\ref{sec:transfer}), then classify the moment functionals generated by derivators (Section~\ref{sec:image}), and finally develop explicit realizations and the complete one-jump model (Sections~\ref{sec:dictionary} and~\ref{sec:model}).

\begin{definition}
Let $x_{0} \in {\mathbf{R}}$ and $n,m \in \{0,1,\dots \}$. We define
\begin{equation}
g_{x_{0},n}(x) \star g_{x_{0},m}(x):=g_{x_{0},n+m}(x).
\end{equation}
It is possible to extend this operation to ${\mathbb P}_{g}$ as
\begin{equation}
{\left( \sum_{i=0}^{n} a_{i} g_{x_{0},i}(x) \right) \star \left( \sum_{k=0}^{m} b_{k} g_{x_{0},k}(x) \right) := \sum_{\ell=0}^{n+m}\left(\sum_{i=\max\{0,\ell-m\}}^{\min\{n,\ell\}} a_i b_{\ell-i}\right)g_{x_0,\ell}(x).}
\end{equation}
\end{definition}

From this point onward we write
\begin{equation}\label{eq:star-bilinear-form}
\langle u,v\rangle_S:=\int_{[a,b)}(u\star v)\,d\mu_g.
\end{equation}
Because $\star$ is bilinear and commutative, $\langle\cdot,\cdot\rangle_S$ is a symmetric bilinear form on $\mathbb P_g$.  It is important not to call it an inner product at this level of generality: Section~\ref{sec:image} proves that it need not be positive semidefinite.  We reserve the term ``inner product'' for those cases in which the associated moment functional is positive definite.

\begin{definition}
Let us consider the basis of $g$-monomials introduced in Definition \ref{def:gmonomials}. Let
\begin{equation}\label{eq:expansion}
P_{n}(x;g)=c_{n,n} g_{x_{0},n}(x)+\sum_{j=0}^{n-1} c_{n,j} g_{x_{0},j}(x), \quad n=0,1,2,\dots,
\end{equation}
be a sequence of $g$-polynomials. We say that $\{P_{n}(x;g)\}_{n}$ is a \emph{sequence of $g$-orthogonal $g$-polynomials} on $[a,b)$ if there exist nonzero real constants $h_n$ such that
\begin{equation}\label{eq:star-orthogonality-general}
\langle P_{n}(x;g),P_{m}(x;g) \rangle_{S}=\int_{[a,b)} \left( P_{n}(s;g) \star P_{m}(s;g) \right) d\mu_{g}(s)=h_n \delta_{n,m}.
\end{equation}
No sign is imposed on $h_n$ in the quasi-definite but non-positive case.  When the form \eqref{eq:star-bilinear-form} is positive definite, each $h_n$ is positive and the usual normalization $h_n=1$ defines a sequence of \emph{$g$-orthonormal $g$-polynomials}.  In the explicit positive model of Section~\ref{sec:model} we retain the historical notation $d_n^2=h_n$; outside the positive region, a superscript ``2'' must be regarded only as notation and not as an assertion that the constant is positive.

We say that $\{P_{n}(x;g)\}_{n}$ is a sequence of \emph{monic $g$-polynomials} if $c_{n,n}=1$. In this case, we shall use the notation $\{\hat{P}_{n}(x;g)\}_{n}$.
\end{definition}

\section{The transfer principle}\label{sec:transfer}

The construction of Section~\ref{sec:gogp} raises an immediate structural question: is the $\star$-orthogonality a genuinely new orthogonality, or is it a classical orthogonality written in an unusual basis? The purpose of this section is to answer that question completely. The answer is the second alternative, and we state it in the strongest form we can, because it is precisely this identification that makes the image analysis of Section~\ref{sec:image} possible.

Throughout this section $g$ is a derivator, $x_{0}\in\mathbf{R}$ is a fixed centre, and $[a,b)$ is the interval of orthogonality. We work under the following hypothesis.

\begin{definition}\label{def:admissible}
We say that the pair $(g,x_{0})$ is \emph{admissible} if the family of $g$-monomials $\{g_{x_{0},n}\}_{n\geq0}$ is linearly independent in the space of real functions.
\end{definition}

Admissibility is not a technical convenience: it is a transparent sufficient condition under which the coefficientwise definition of the $\star$-product is unambiguous and under which the transfer map below is an isomorphism. Without linear independence, a function may admit several coefficient representations, and the bilinear prescription is not automatically independent of that choice; one would then have to pass to an appropriate quotient and prove separately that the relevant kernel is an ideal. We do not pursue that quotient construction here. Admissibility is, however, a mild restriction in the situations considered below, as the following observation shows.

\begin{proposition}\label{prop:admissible}
Let $g$ be a derivator and $x_{0}\in\mathbf{R}$, and suppose that the $g$-derivative of every $g_{x_{0},n}$ is defined at every point of the domain under consideration, so that \eqref{eq:dgm} holds there. Then $(g,x_{0})$ is admissible.
\end{proposition}

\begin{proof}
Suppose $\sum_{n=0}^{N}\lambda_{n}g_{x_{0},n}=0$ with $\lambda_{N}\neq0$. Apply $D_g$ successively. For every index $n<N$, after $n+1$ applications one reaches
\[
D_g^{n+1}g_{x_0,n}=D_g(n!g_{x_0,0})=0,
\]
so that term remains identically zero under all subsequent applications. The top-degree term survives through the first $N$ differentiations and satisfies
\[
D_g^N g_{x_0,N}=N!g_{x_0,0}=N!.
\]
Consequently, after applying $D_g^N$ to the assumed relation, all terms of degree below $N$ have vanished and we obtain
\[
0=\lambda_N N!.
\]
This contradicts $\lambda_N\neq0$. Hence no nontrivial finite linear relation among the $g$-monomials exists, and $(g,x_0)$ is admissible.
\end{proof}

Proposition~\ref{prop:admissible} is a sufficient criterion for admissibility; it is not intended as a classification of all possible failures of linear independence.  In the sequel admissibility is always stated explicitly.  In the concrete derivators used below it follows either from Proposition~\ref{prop:admissible} or directly from the existence of an interval on which the restrictions of $g_{x_0,n}$ have distinct polynomial degrees.

\begin{theorem}[Transfer principle: algebraic part]\label{thm:transfer-alg}
Let $(g,x_{0})$ be admissible and let
\begin{equation}\label{eq:Phi}
\Phi:\mathbf{R}[t]\longrightarrow\mathbb{P}_{g},\qquad
\Phi\Bigl(\sum_{k=0}^{N}a_{k}t^{k}\Bigr)=\sum_{k=0}^{N}a_{k}\,g_{x_{0},k}.
\end{equation}
Then $\Phi$ is an isomorphism of unital commutative algebras from $(\mathbf{R}[t],\cdot)$ onto $(\mathbb{P}_{g},\star)$, it preserves degrees and leading coefficients, and it intertwines the ordinary derivative with the Stieltjes derivative:
\begin{equation}\label{eq:intertwine}
\Phi(pq)=\Phi(p)\star\Phi(q),\qquad
D_{g}\Phi(p)=\Phi(p'),\qquad p,q\in\mathbf{R}[t].
\end{equation}
Equivalently, $\Phi^{-1}D_{g}\Phi=d/dt$, so that $(\mathbb{P}_{g},\star,D_{g})$ and $(\mathbf{R}[t],\cdot,d/dt)$ are isomorphic as differential algebras.
\end{theorem}

\begin{proof}
The map $\Phi$ is surjective by the definition of $\mathbb{P}_{g}$ and injective by admissibility, hence a linear isomorphism; it sends $t^{n}$ to a $g$-polynomial of degree $n$ with leading coefficient $1$, so it preserves degrees and leading coefficients. On monomials,
\[
\Phi(t^{i}t^{j})=\Phi(t^{i+j})=g_{x_{0},i+j}=g_{x_{0},i}\star g_{x_{0},j}=\Phi(t^{i})\star\Phi(t^{j}),
\]
by the definition of $\star$, and the first identity in \eqref{eq:intertwine} follows by bilinearity. For the second, \eqref{eq:dgm} gives $(g_{x_{0},n})'_{g}=n\,g_{x_{0},n-1}$, whence
\[
D_{g}\Phi(t^{n})=n\,g_{x_{0},n-1}=\Phi\bigl(nt^{n-1}\bigr)=\Phi\Bigl(\frac{d}{dt}t^{n}\Bigr),
\]
and again linearity concludes.
\end{proof}

The analytic half of the transfer principle identifies the $\star$-bilinear form. Its proof is nothing but the fundamental theorem of calculus for Stieltjes integrals, but the resulting formula is the tool on which the image theorem of Section~\ref{sec:image} rests, so we isolate it.

\begin{theorem}[Transfer principle: analytic part]\label{thm:transfer-an}
Let $(g,x_{0})$ be admissible, let $a<b$ with $a,b\notin C_{g}$, and set
\begin{equation}\label{eq:master-moments}
m_{n}:=\int_{[a,b)}g_{x_{0},n}\,d\mu_{g},\qquad n\geq0.
\end{equation}
Then
\begin{equation}\label{eq:master}
m_{n}=\frac{g_{x_{0},n+1}(b)-g_{x_{0},n+1}(a)}{n+1},\qquad n\geq0,
\end{equation}
and, denoting by $\mathcal{L}$ the moment functional on $\mathbf{R}[t]$ defined by $\mathcal{L}(t^{n})=m_{n}$,
\begin{equation}\label{eq:pullback}
\langle \Phi(p),\Phi(q)\rangle_{S}=\mathcal{L}(pq),\qquad p,q\in\mathbf{R}[t].
\end{equation}
\end{theorem}

\begin{proof}
By Definition~\ref{def:gmonomials}, $g_{x_{0},n+1}(x)=(n+1)\int_{x_{0}}^{x}g_{x_{0},n}\,d\mu_{g}$ for every $x$. Using the additivity convention of Definition~\ref{def:mug} for the oriented integrals,
\[
\int_{[a,b)}g_{x_{0},n}\,d\mu_{g}
=\int_{x_{0}}^{b}g_{x_{0},n}\,d\mu_{g}-\int_{x_{0}}^{a}g_{x_{0},n}\,d\mu_{g}
=\frac{g_{x_{0},n+1}(b)-g_{x_{0},n+1}(a)}{n+1},
\]
which is \eqref{eq:master}. For \eqref{eq:pullback} it suffices, by bilinearity, to take $p=t^{i}$, $q=t^{j}$; then
\[
\langle\Phi(t^{i}),\Phi(t^{j})\rangle_{S}
=\int_{[a,b)}\bigl(g_{x_{0},i}\star g_{x_{0},j}\bigr)\,d\mu_{g}
=\int_{[a,b)}g_{x_{0},i+j}\,d\mu_{g}=m_{i+j}=\mathcal{L}(t^{i}t^{j}).\qedhere
\]
\end{proof}

The following corollary is the precise statement of what the $\star$-construction does and does not produce. We write $\{\hat q_{n}\}_{n}$ for the sequence of monic orthogonal polynomials associated with $\mathcal{L}$, when it exists.

\begin{corollary}\label{cor:transfer}
Let $(g,x_{0})$ be admissible and let $\mathcal{L}$ be as in Theorem~\ref{thm:transfer-an}. Then:
\begin{enumerate}
\item A sequence of monic $\star$-orthogonal $g$-polynomials exists if and only if $\mathcal{L}$ is quasi-definite, i.e. if and only if the Hankel determinants $\widehat\Delta_{n}=\det(m_{i+j})_{i,j=0}^{n-1}$ are nonzero for all $n\geq1$. In that case it is unique and given by $\hat P_{n}(\cdot\,;g)=\Phi(\hat q_{n})$.
\item The orthogonality constants are transported unchanged: if $h_n:=\langle\hat P_n,\hat P_n\rangle_S$, then $h_n=\mathcal{L}(\hat q_n^2)$.  In the positive-definite case these are the squared norms; in an indefinite quasi-definite case they are merely nonzero real diagonal constants.
\item If $t\hat q_{n}=\hat q_{n+1}+\beta_{n}\hat q_{n}+\gamma_{n}\hat q_{n-1}$ is the classical three-term recurrence, then
\begin{equation}\label{eq:ttrr-transported}
g_{x_{0},1}\star\hat P_{n}(\cdot\,;g)=\hat P_{n+1}(\cdot\,;g)+\beta_{n}\hat P_{n}(\cdot\,;g)+\gamma_{n}\hat P_{n-1}(\cdot\,;g),
\end{equation}
with the same coefficients $\beta_{n},\gamma_{n}$.
\item If $\sigma,\tau,\omega\in\mathbf{R}[t]$ and $\Sigma=\Phi(\sigma)$, $T=\Phi(\tau)$, $W=\Phi(\omega)$, then
\begin{equation}\label{eq:gde-transported}
\Sigma\star D_{g}^{2}u+T\star D_{g}u+W\star u=0
\quad\Longleftrightarrow\quad
\sigma\,v''+\tau\,v'+\omega\,v=0,
\end{equation}
where $u=\Phi(v)$. In particular a second-order $g$-differential equation for the $\star$-orthogonal $g$-polynomials is equivalent to an ordinary second-order differential equation for the ordinary polynomial family $\{\hat q_{n}\}$, of the same degrees.
\end{enumerate}
\end{corollary}

\begin{proof}
All four items are immediate from Theorems~\ref{thm:transfer-alg} and~\ref{thm:transfer-an}. For (i), a monic $g$-polynomial $\hat P_{n}(\cdot\,;g)$ is $\star$-orthogonal to $\mathbb{P}_{g}$ of degree $<n$ if and only if $\Phi^{-1}\hat P_{n}(\cdot\,;g)$ is a monic polynomial orthogonal with respect to $\mathcal{L}$ to all polynomials of degree $<n$, by \eqref{eq:pullback} and the fact that $\Phi$ preserves degrees; the classical existence and uniqueness criterion \cite{MR0481884} then applies. Item (ii) is \eqref{eq:pullback} with $p=q=\hat q_{n}$; the additional terminology about squared norms follows exactly when positivity makes both sides positive for nonzero polynomials. For (iii), apply $\Phi$ to the classical recurrence and use $\Phi(t\hat q_{n})=g_{x_{0},1}\star\Phi(\hat q_{n})$. For (iv), apply $\Phi$ to the ordinary equation and use \eqref{eq:intertwine} twice; injectivity of $\Phi$ gives the converse implication.
\end{proof}

\begin{remark}\label{rem:honest}
Corollary~\ref{cor:transfer} should be read as a negative result and a positive one at the same time. On the negative side, the $\star$-construction never produces a new orthogonal family in the algebraic or moment-functional sense: as an abstract differential algebra, $(\mathbb{P}_{g},\star,D_{g})$ is the ordinary polynomial algebra, and every $\star$-orthogonal family is the $\Phi$-image of an ordinary polynomial family associated with the moment functional $\mathcal L$. Any claim of novelty formulated at that level would be incorrect. On the positive side, the whole construction is condensed into the single functional $\mathcal{L}$ of \eqref{eq:master}, which is the object that retains all information relevant to \(\star\)-orthogonality, and Sections~\ref{sec:image} and~\ref{sec:dictionary} show that determining and classifying it is a nontrivial problem. What remains specific to the Stieltjes setting is therefore twofold: the realisation of ordinary polynomial families as explicit functions of the original variable $x$, expressed in the $g$-monomial basis and differentiated with $D_{g}$, and the correspondence $g\mapsto\mathcal{L}$ itself.
\end{remark}

\begin{remark}[the $\star$-product depends on the centre]\label{rem:centre}
The linear space $\mathbb{P}_{g}$ does not depend on the centre $x_{0}$ \cite[Proposition 3.14]{Cora}, but the product $\star$ does, since it is defined through the basis $\{g_{x_{0},n}\}_{n}$. Consequently the functional $\mathcal{L}$ of \eqref{eq:master} depends on $x_{0}$ as well. In the classical case $g=g_{0}$ the change of centre is the substitution $t\mapsto t+(x_{0}-x_{0}')$, an algebra automorphism of $\mathbf{R}[t]$, so nothing happens; for a general derivator the $g$-monomials do not satisfy a binomial theorem and the two products are not conjugate by an affine change of variable. Theorems~\ref{thm:image} and~\ref{thm:left} below show more precisely what changes with the centre: a jump that lies to the right of the centre contributes an endpoint derivative term, whereas the same jump placed to the left contributes an absolutely continuous exponential tail.  Compactness of the representing measure may therefore change under a change of centre.  The centre is consequently part of the data of the construction, not a harmless normalization.
\end{remark}

\section{The image of the correspondence $g\mapsto\mathcal{L}$}\label{sec:image}

The transfer principle reduces the unweighted \(\star\)-orthogonality to the functional \(\mathcal L\).  We now characterize exactly which functionals arise when the centre lies in the orthogonality interval.  The proof has three parts: a generating identity, an explicit representation, and an inverse argument showing that the parameters occurring in that representation are uniquely determined by the functional.

\subsection{The Stieltjes exponential and the master generating identity}

Throughout this section \((g,x_0)\) is admissible, \(a<b\), and \(x_0\in[a,b]\).  Write \(g=g^C+g^J\) for the continuous and jump parts on \([a,b]\).  For \(z\) sufficiently close to the origin, the solution of
\[
 y'_g=zy,\qquad y(x_0)=1,
\]
is the Stieltjes exponential \cite{FRAN,RODRIGO}
\begin{equation}\label{eq:stieltjes-exp}
 e_z(x)=
 \begin{dcases}
 \exp\!\bigl(z(g^C(x)-g^C(x_0))\bigr)
 \displaystyle\prod_{s\in D_g\cap[x_0,x)}(1+z\Delta g(s)),&x\ge x_0,\\[2mm]
 \exp\!\bigl(z(g^C(x)-g^C(x_0))\bigr)
 \displaystyle\prod_{s\in D_g\cap[x,x_0)}(1+z\Delta g(s))^{-1},&x<x_0.
 \end{dcases}
\end{equation}
On a bounded interval the jump sizes are summable, so the products converge normally on compact subsets that avoid the poles of the reciprocal factors.

\begin{lemma}[Exponential generating function of the $g$-monomials]\label{lem:egf}
For every fixed \(x\) there exists \(r_x>0\) such that, for \(|z|<r_x\),
\begin{equation}\label{eq:egf}
 \sum_{n=0}^{\infty}g_{x_0,n}(x)\frac{z^n}{n!}=e_z(x).
\end{equation}
\end{lemma}

\begin{proof}
Set \(I_x=[\min\{x,x_0\},\max\{x,x_0\})\) and \(M_x=\mu_g(I_x)\).  The recursive definition of the Stieltjes monomials implies, by induction,
\[
 |g_{x_0,n}(y)|\le n!\,M_x^n,
 \qquad y\in I_x\cup\{x,x_0\}.
\]
Indeed the assertion is trivial for \(n=0\); if it holds at level \(n-1\), then the oriented-integral definition and positivity of \(\mu_g\) give
\[
 |g_{x_0,n}(y)|
 \le n\int_{I_x}|g_{x_0,n-1}(s)|\,d\mu_g(s)
 \le n!M_x^n.
\]
Hence the series in \eqref{eq:egf} converges absolutely and uniformly, simultaneously in the spatial variable on \(I_x\cup\{x,x_0\}\), for \(|z|<M_x^{-1}\) when \(M_x>0\) (and trivially when \(M_x=0\)).  We avoid any unproved interchange of \(D_g\) with an infinite series and return instead to the defining integral recursion.  If
\[
 E(y,z):=\sum_{n=0}^{\infty}g_{x_0,n}(y)\frac{z^n}{n!},
\]
then uniform absolute convergence and the finiteness of \(\mu_g(I_x)\) allow summation under the Stieltjes integral.  For \(y\ge x_0\),
\begin{align*}
 E(y,z)
 &=1+\sum_{n\ge1}\frac{z^n}{(n-1)!}
       \int_{[x_0,y)}g_{x_0,n-1}(s)\,d\mu_g(s)\\
 &=1+z\int_{[x_0,y)}E(s,z)\,d\mu_g(s).
\end{align*}
For \(y<x_0\) the same computation with the oriented integral gives the identical Volterra equation with the corresponding orientation.  Thus \(E(\cdot,z)\) satisfies the integral form of
\[
 D_gE=zE,\qquad E(x_0,z)=1.
\]
The Stieltjes exponential \(e_z\) in \eqref{eq:stieltjes-exp} is the unique solution of this linear integral equation.  Consequently \(E(y,z)=e_z(y)\), and in particular \eqref{eq:egf} holds at the prescribed point \(x\).
\end{proof}

For a linear functional on polynomials we use the notation
\[
 \mathcal L(e^{zt}):=\sum_{n\ge0}\mathcal L(t^n)\frac{z^n}{n!}
\]
whenever this exponential generating series converges.

\begin{theorem}[Master generating identity]\label{thm:master-gf}
Let \(m_n\) be the moments in \eqref{eq:master-moments}.  Then, for \(z\) in a neighbourhood of \(0\),
\begin{equation}\label{eq:mastergf}
 \mathcal L(e^{zt})
 =\sum_{n\ge0}m_n\frac{z^n}{n!}
 =\frac{e_z(b)-e_z(a)}{z},
\end{equation}
where the value at \(z=0\) is understood by removable continuation.
\end{theorem}

\begin{proof}
The master moment formula \eqref{eq:master} gives
\[
 \frac{m_n}{n!}
 =\frac{g_{x_0,n+1}(b)-g_{x_0,n+1}(a)}{(n+1)!}.
\]
Therefore, for \(z\ne0\) sufficiently small,
\begin{align*}
 \sum_{n\ge0}m_n\frac{z^n}{n!}
 &=\frac1z\sum_{n\ge0}
 \bigl(g_{x_0,n+1}(b)-g_{x_0,n+1}(a)\bigr)
 \frac{z^{n+1}}{(n+1)!}\\
 &=\frac1z\left[
 \sum_{k\ge1}g_{x_0,k}(b)\frac{z^k}{k!}
 -\sum_{k\ge1}g_{x_0,k}(a)\frac{z^k}{k!}
 \right]\\
 &=\frac{(e_z(b)-1)-(e_z(a)-1)}{z}
 =\frac{e_z(b)-e_z(a)}{z},
\end{align*}
where Lemma~\ref{lem:egf} was used at the two endpoints.  The numerator vanishes at \(z=0\), so the apparent singularity is removable.
\end{proof}

\subsection{The image theorem}

Define the nonnegative continuous increments
\begin{equation}\label{eq:cpm}
 c_-:=g^C(x_0)-g^C(a),\qquad
 c_+:=g^C(b)-g^C(x_0),
\end{equation}
and the at most countable multisets
\begin{equation}\label{eq:AB}
 A:=\{\Delta g(s):s\in D_g\cap[x_0,b)\},\qquad
 B:=\{\Delta g(s):s\in D_g\cap[a,x_0)\}.
\end{equation}
All their elements are positive and
\(\sum_{\alpha\in A}\alpha<\infty\),
\(\sum_{\beta\in B}\beta<\infty\).  Put
\begin{equation}\label{eq:PiAB}
 \Pi_A(z):=\prod_{\alpha\in A}(1+\alpha z)
 =\sum_{j\ge0}e_j(A)z^j,
 \qquad
 \Pi_B(z):=\prod_{\beta\in B}(1+\beta z),
\end{equation}
with \(e_0(A)=1\).  The coefficients \(e_j(A)\) are finite because
\(0\le e_j(A)\le (\sum_{\alpha\in A}\alpha)^j/j!\).

For each \(\beta\in B\), let \(E_\beta\) be an independent exponential random variable with mean \(\beta\), and set
\[
 X_B:=\sum_{\beta\in B}E_\beta.
\]
Since \(\sum_{\beta\in B}\mathbb E E_\beta=\sum_{\beta\in B}\beta<\infty\), monotone convergence gives \(\mathbb E X_B<\infty\), and hence \(X_B<\infty\) almost surely.  Let
\begin{equation}\label{eq:survival}
 S_B(v):=\mathbb P(X_B>v),\qquad v\ge0,
\end{equation}
with \(S_\emptyset\equiv0\).  For \(z\) in a sufficiently small complex neighbourhood of the origin,
\begin{equation}\label{eq:laplaceXB}
 \mathbb E(e^{-zX_B})=\frac1{\Pi_B(z)}.
\end{equation}
Indeed, if \(X_{B,m}\) denotes a finite partial sum, independence gives
\(\mathbb E e^{-zX_{B,m}}=\prod_{j\le m}(1+\beta_jz)^{-1}\).  Since \(\sum\beta_j<\infty\), one can choose \(\delta>0\) such that \(\delta\sup_j\beta_j<1/2\) and
\[
 \sup_m\mathbb E e^{\delta X_{B,m}}
 =\sup_m\prod_{j\le m}(1-\delta\beta_j)^{-1}<\infty.
\]
The partial sums increase almost surely to \(X_B\), and the preceding exponential-moment bound gives uniform integrability for \(|\Re z|<\delta\).  Passing to the limit proves \eqref{eq:laplaceXB}; normal convergence of \(\Pi_B\) gives the same analytic function on that neighbourhood.
If \(B=\{\beta_1,\ldots,\beta_m\}\) is finite with pairwise distinct entries, the familiar hypoexponential formula gives
\begin{equation}\label{eq:survival-explicit}
 S_B(v)=\sum_{j=1}^m
 \frac{\beta_j^{m-1}}{\prod_{i\ne j}(\beta_j-\beta_i)}e^{-v/\beta_j}.
\end{equation}
The probabilistic definition \eqref{eq:survival} is used when entries repeat or \(B\) is infinite.

\begin{theorem}[Image and structure theorem]\label{thm:image}
Let \((g,x_0)\) be admissible and \(x_0\in[a,b]\).  With the notation above,
\begin{equation}\label{eq:mastergf2}
 {\displaystyle
 \mathcal L(e^{zt})=
 \frac1z\left(
 e^{c_+z}\Pi_A(z)-\frac{e^{-c_-z}}{\Pi_B(z)}
 \right)}
\end{equation}
for \(z\) in a neighbourhood of the origin.  Equivalently, for every \(p\in\mathbb R[t]\),
\begin{equation}\label{eq:structure}
 {\displaystyle
 \mathcal L(p)=
 \int_{-c_-}^{c_+}p(t)\,dt
 +\sum_{j\ge1}e_j(A)p^{(j-1)}(c_+)
 +\int_{-\infty}^{-c_-}p(t)S_B(-c_--t)\,dt.}
\end{equation}
The middle sum is finite for each polynomial \(p\).

Conversely, for every quadruple \((c_-,c_+,A,B)\) with \(c_\pm\ge0\) and summable multisets \(A,B\) of positive numbers, there exist a derivator \(g\), an interval \([a,b]\), and a centre \(x_0\in(a,b)\) such that \((g,x_0)\) is admissible and its associated functional is \eqref{eq:structure}.  Hence \eqref{eq:structure} describes the full image of the normalized correspondence \((g,x_0;[a,b])\mapsto\mathcal L\) with \(x_0\in[a,b]\).
\end{theorem}

\begin{proof}
By \eqref{eq:stieltjes-exp}, evaluation at the right endpoint gives
\[
 e_z(b)=e^{c_+z}\Pi_A(z),
\]
whereas evaluation at the left endpoint gives
\[
 e_z(a)=e^{-c_-z}\Pi_B(z)^{-1}.
\]
Substitution in Theorem~\ref{thm:master-gf} proves \eqref{eq:mastergf2}.

We next derive \eqref{eq:structure} without any formal inversion.  Split the right-hand side of \eqref{eq:mastergf2} into
\begin{equation}\label{eq:splitmaster}
 \frac{e^{c_+z}-e^{-c_-z}}{z}
 +e^{c_+z}\frac{\Pi_A(z)-1}{z}
 +e^{-c_-z}\frac{1-\Pi_B(z)^{-1}}{z}.
\end{equation}
The first term equals
\[
 \int_{-c_-}^{c_+}e^{zt}\,dt.
\]
For the second term, normal convergence of \(\Pi_A\) gives
\[
 e^{c_+z}\frac{\Pi_A(z)-1}{z}
 =\sum_{j\ge1}e_j(A)z^{j-1}e^{c_+z}.
\]
Since
\[
 z^{j-1}e^{c_+z}
 =\left.\frac{d^{j-1}}{dt^{j-1}}e^{zt}\right|_{t=c_+},
\]
comparison of exponential generating coefficients produces the middle term of \eqref{eq:structure}.  For a polynomial of degree \(d\), derivatives of order greater than \(d\) vanish, so this contribution is always a finite sum even when \(A\) is infinite.

For the third term, use \eqref{eq:laplaceXB}.  For real \(z>0\) sufficiently small, Tonelli's theorem gives
\begin{align*}
 \frac{1-\Pi_B(z)^{-1}}{z}
 &=\frac1z\int_0^\infty(1-e^{-zv})\,d\mathbb P_{X_B}(v)\\
 &=\int_0^\infty\left(\int_0^v e^{-zu}\,du\right)d\mathbb P_{X_B}(v)\\
 &=\int_0^\infty e^{-zu}\,\mathbb P(X_B>u)\,du
 =\int_0^\infty e^{-zu}S_B(u)\,du.
\end{align*}
Both sides are analytic near \(z=0\), so the identity extends there.  Multiplication by \(e^{-c_-z}\) and the change of variables \(t=-c_--u\) yield
\[
 e^{-c_-z}\frac{1-\Pi_B(z)^{-1}}{z}
 =\int_{-\infty}^{-c_-}e^{zt}S_B(-c_--t)\,dt.
\]
Combining the three pieces of \eqref{eq:splitmaster} and comparing Taylor coefficients proves \eqref{eq:structure} on monomials and hence on all polynomials.

It remains to prove the converse, including admissibility.  Fix, for definiteness, \(a=-1\), \(x_0=0\), \(b=1\).  Define a continuous non-decreasing function \(h\) on \([-1,1]\) by
\[
 h(x)=
 \begin{cases}
 c_-x,&-1\le x\le0,\\
 c_+x,&0\le x\le1.
 \end{cases}
\]
Thus \(h(0)-h(-1)=c_-\) and \(h(1)-h(0)=c_+\); either side is allowed to be flat.  Enumerate \(A=\{\alpha_i\}\) and \(B=\{\beta_j\}\).  Choose distinct points \(u_i\in[0,1)\) increasing to \(1\) when \(A\) is infinite, and distinct points \(v_j\in(-1,0)\) decreasing to \(-1\) when \(B\) is infinite.  Set on \([-1,1]\)
\[
 g(x)=h(x)
 +\sum_i\alpha_i\chi_{(u_i,\infty)}(x)
 +\sum_j\beta_j\chi_{(v_j,\infty)}(x).
\]
Summability of the jumps makes this a finite-valued left-continuous non-decreasing function.  Extend it continuously and strictly increasingly outside \([-1,1]\), for example with slope one, preserving its endpoint values.  The resulting global function is a derivator with precisely the prescribed data on the orthogonality interval.

To verify admissibility, consider the open interval \((1,2)\).  There are no jumps there and the extension has slope one, so \(D_g\) is the ordinary derivative on that interval.  From \(D_gg_{0,n}=n g_{0,n-1}\), an induction on \(n\) shows that the restriction of \(g_{0,n}\) to \((1,2)\) is an ordinary polynomial in \(x\) of degree \(n\) with leading coefficient one: the assertion is clear for \(n=0\), and integrating the derivative relation raises the degree by one while preserving the monic leading coefficient.  Hence a relation \(\sum_{n=0}^N\lambda_ng_{0,n}=0\) with \(\lambda_N\ne0\) would restrict on \((1,2)\) to an ordinary polynomial of degree \(N\) with leading coefficient \(\lambda_N\), which cannot vanish identically.  Thus \((g,0)\) is admissible.  Formula \eqref{eq:mastergf2} then gives the prescribed functional.
\end{proof}

\subsection{Exact description of the fibres}

Theorem~\ref{thm:image} proves that \((c_-,c_+,A,B)\) determines \(\mathcal L\).  The converse is less immediate and is needed if the word ``fibres'' is to be understood literally.

\begin{theorem}[Uniqueness of the structural parameters]\label{thm:fibres}
Let \(\mathcal L_{c_-,c_+,A,B}\) and \(\mathcal L_{\widetilde c_-,\widetilde c_+,\widetilde A,\widetilde B}\) be two functionals of the form \eqref{eq:structure}.  If they agree on \(\mathbb R[t]\), then
\[
 c_-=\widetilde c_-,\qquad
 c_+=\widetilde c_+,\qquad
 A=\widetilde A,\qquad B=\widetilde B
\]
as multisets.  Consequently the fibres of the correspondence are exactly the classes of derivators having the same quadruple \((c_-,c_+,A,B)\).
\end{theorem}

\begin{proof}
Equality of the functionals means equality of all moments, hence equality of their exponential generating functions in a neighbourhood of the origin.  Put
\begin{equation}\label{eq:Fmeromorphic}
 F(z):=z\mathcal L(e^{zt})
 =e^{c_+z}\Pi_A(z)-\frac{e^{-c_-z}}{\Pi_B(z)}.
\end{equation}
The right-hand side is meromorphic on \(\mathbb C\), with possible poles only on the negative real axis.  Equality near \(0\) therefore implies equality of the meromorphic continuations.

We first recover \(B\).  The entire function \(e^{c_+z}\Pi_A(z)\) cannot cancel a pole of the second term.  Thus the poles of \(F\), with multiplicities, are exactly
\[
 z=-\beta^{-1},\qquad \beta\in B,
\]
and they determine the multiset \(B\) uniquely.  The same argument gives \(B=\widetilde B\).

Assume first that \(B\ne\emptyset\), and choose \(\beta\in B\) of multiplicity \(m\).  Let
\[
 \Pi_B(z)=(1+\beta z)^m Q_\beta(z),
 \qquad Q_\beta(-1/\beta)\ne0.
\]
Multiplying \eqref{eq:Fmeromorphic} by \((1+\beta z)^m\) and taking \(z\to-1/\beta\) eliminates the entire first term and gives
\[
 \lim_{z\to-1/\beta}(1+\beta z)^mF(z)
 =-\frac{e^{c_-/\beta}}{Q_\beta(-1/\beta)}.
\]
Everything on the right except \(c_-\) is already known from \(B\).  Since \(c_-\mapsto e^{c_-/\beta}\) is injective on \([0,\infty)\), this determines \(c_-\).

If \(B=\emptyset\), there are no poles and
\[
 F(-r)=e^{-c_+r}\Pi_A(-r)-e^{c_-r},\qquad r>0.
\]
The canonical product \(\Pi_A\) has zero exponential type.  Indeed,
\[
 0\le\frac{\log\Pi_A(r)}r
 =\sum_{\alpha\in A}\frac{\log(1+\alpha r)}r
 \le\sum_{\alpha\in A}\alpha,
\]
and for each fixed \(\alpha\), \(r^{-1}\log(1+\alpha r)\to0\); dominated convergence over the summable multiset \(A\) yields
\[
 \lim_{r\to\infty}\frac{\log\Pi_A(r)}r=0.
\]
Moreover \(|\Pi_A(-r)|\le\Pi_A(r)\), so the same upper exponential estimate holds on the negative axis.  It follows that
\begin{equation}\label{eq:recover-cminus}
 c_-=\limsup_{r\to\infty}\frac1r\log^+|F(-r)|,
 \qquad \log^+x:=\max\{0,\log x\}\ \ (x>0),\qquad \log^+0:=0.
\end{equation}
If \(c_->0\), the term \(-e^{c_-r}\) dominates the subexponential first term; if \(c_-=0\), the limsup is zero.  Thus \(c_-\) is again determined.

Having recovered \(B\) and \(c_-\), form
\begin{equation}\label{eq:Hrecover}
 H(z):=F(z)+\frac{e^{-c_-z}}{\Pi_B(z)}=e^{c_+z}\Pi_A(z).
\end{equation}
For \(r>0\), \(H(r)>0\).  The same zero-exponential-type estimate gives
\begin{equation}\label{eq:recover-cplus}
 c_+=\lim_{r\to\infty}\frac{\log H(r)}r.
\end{equation}
Finally the zeros of \(H\) are precisely \(-1/\alpha\), \(\alpha\in A\), with the same multiplicities, because the exponential factor never vanishes.  Therefore \(A\) is uniquely recovered.  This proves uniqueness of all four components.
\end{proof}

\begin{corollary}[Fibres]\label{cor:fibres}
Neither the positions of the jumps nor the detailed shape of the continuous part of \(g\) affect \(\mathcal L\).  Two admissible centered derivators with centres in their orthogonality intervals yield the same functional if and only if they have the same two continuous increments and the same multisets of right- and left-hand jump sizes.
\end{corollary}

\subsection{Positivity and positive definiteness}

The structural formula contains a positive interval part and a positive left-tail part, but the right-hand derivative evaluations need not be positive.  The next theorem gives the exact boundary.

\begin{theorem}[Positive-measure representation]\label{thm:positivity}
Let \(\mathcal L=\mathcal L_{c_-,c_+,A,B}\) be in the image described by Theorem~\ref{thm:image}.  Then the following are equivalent:
\begin{enumerate}
\item \(\mathcal L(p^2)\ge0\) for every \(p\in\mathbb R[t]\);
\item \(\mathcal L\) is the moment functional of a finite positive Borel measure on \(\mathbb R\);
\item \(|A|\le1\).
\end{enumerate}
If \(A=\emptyset\), put \(\alpha=0\); if \(A=\{\alpha\}\), retain that value.  In these cases a representing measure is
\begin{equation}\label{eq:nu}
 d\nu(t)=
 \mathbf 1_{[-c_-,c_+]}(t)\,dt
 +\alpha\,\delta_{c_+}
 +S_B(-c_--t)\mathbf 1_{(-\infty,-c_-]}(t)\,dt.
\end{equation}
Moreover, \(\mathcal L\) is positive definite, i.e. \(\mathcal L(p^2)>0\) for every nonzero polynomial \(p\), if and only if
\begin{equation}\label{eq:pdcriterion}
 |A|\le1
 \quad\text{and}\quad
 \bigl(c_-+c_+>0\ \text{or}\ B\ne\emptyset\bigr).
\end{equation}
\end{theorem}

\begin{proof}
The implication (ii)\(\Rightarrow\)(i) is immediate.  If \(|A|\le1\), formula \eqref{eq:structure} is exactly integration against \eqref{eq:nu}; the total mass is finite because
\[
 \int_0^\infty S_B(v)\,dv=\mathbb E X_B=\sum_{\beta\in B}\beta<\infty.
\]
Thus (iii)\(\Rightarrow\)(ii).

It remains to prove that two or more right-hand jumps are incompatible with positivity.  Assume, for contradiction, that \(|A|\ge2\) and \(\mathcal L(p^2)\ge0\) for every polynomial.  By the Hamburger moment theorem, there exists a finite positive Borel measure \(\mu\) whose moments are those of \(\mathcal L\).  We next justify that \(\mu\) has an exponential moment; this step is needed in order to identify a Laplace transform rather than only a formal moment series.

The function \(M(z):=\mathcal L(e^{zt})\) in \eqref{eq:mastergf2} is analytic in a disc \(|z|<R\) for some \(R>0\).  Cauchy's estimate therefore gives constants \(C,R_0>0\) such that
\[
 0\le m_{2n}\le C(2n)!R_0^{-2n},\qquad n\ge0.
\]
Choose \(0<\delta<R_0\).  Since \(e^{\delta|t|}\le2\cosh(\delta t)\), Tonelli's theorem yields
\begin{align*}
 \int_{\mathbb R}e^{\delta|t|}\,d\mu(t)
 &\le2\int_{\mathbb R}\cosh(\delta t)\,d\mu(t)\\
 &=2\sum_{n\ge0}\frac{\delta^{2n}}{(2n)!}m_{2n}
 \le2C\sum_{n\ge0}(\delta/R_0)^{2n}<\infty.
\end{align*}
Hence the bilateral Laplace transform of \(\mu\) is holomorphic in a strip containing the imaginary axis.  More explicitly, if \(|\Re z|<\delta\), then \(|e^{zt}|\le e^{\delta|t|}\), so dominated convergence permits termwise integration of the exponential series and gives
\[
 \int_{\mathbb R}e^{zt}\,d\mu(t)
 =\sum_{n=0}^\infty\frac{z^n}{n!}\int_{\mathbb R}t^n\,d\mu(t)
 =\sum_{n=0}^\infty m_n\frac{z^n}{n!}
 =M(z)
\]
for \(|z|\) sufficiently small.  Thus equality of moments really yields equality of analytic Laplace transforms near the origin, not merely equality of formal power series.

Let \(\nu_0\) be the positive measure obtained from \eqref{eq:nu} after setting \(\alpha=0\), namely the interval part plus the left-tail part. We verify directly that \(\nu_0\) has an exponential moment. Choose \(\delta>0\) smaller, if necessary, than the constant used in \eqref{eq:laplaceXB}. The interval contribution is finite because it is supported on the compact interval \([-c_-,c_+]\). For the tail, with the change of variables \(u=-c_--t\),
\[
 \int_{-\infty}^{-c_-} e^{\delta |t|}S_B(-c_--t)\,dt
 \le e^{\delta c_-}\int_0^\infty e^{\delta u}S_B(u)\,du.
\]
Since \(X_B\ge0\), Tonelli's theorem and the identity \(e^{\delta X_B}-1=\int_0^{X_B}\delta e^{\delta u}\,du\) give
\[
 \int_0^\infty e^{\delta u}S_B(u)\,du
 =\frac{\mathbb E(e^{\delta X_B})-1}{\delta}<\infty.
\]
Thus \(\nu_0\) has an exponential moment. The finite signed measure \(\sigma:=\mu-\nu_0\) consequently has a bilateral Laplace transform holomorphic in a strip around the imaginary axis.  From \eqref{eq:mastergf2}, equality of moments near \(z=0\) gives
\begin{equation}\label{eq:RA}
 \int_{\mathbb R}e^{zt}\,d\sigma(t)
 =R_A(z):=e^{c_+z}\frac{\Pi_A(z)-1}{z}
\end{equation}
for \(z\) near \(0\).  The apparent singularity of \(R_A\) at zero is removable, and \(R_A\) is entire.  The identity theorem extends \eqref{eq:RA} throughout the common strip and, in particular, to every \(z=iy\), \(y\in\mathbb R\).  Therefore
\[
 \widehat\sigma(y)=R_A(iy).
\]
The Fourier transform of a finite signed measure is bounded by its total variation, so \(|R_A(iy)|\) must be bounded in \(y\).

Choose two elements \(\alpha_1,\alpha_2\in A\), counted with multiplicity.  Since every factor has modulus at least one on the imaginary axis,
\[
 |\Pi_A(iy)|
 =\prod_{\alpha\in A}\sqrt{1+\alpha^2y^2}
 \ge\sqrt{1+\alpha_1^2y^2}\sqrt{1+\alpha_2^2y^2}.
\]
Hence \(|\Pi_A(iy)|\) grows at least quadratically, and
\[
 |R_A(iy)|
 =\frac{|\Pi_A(iy)-1|}{|y|}
 \ge\frac{|\Pi_A(iy)|-1}{|y|}\longrightarrow\infty.
\]
This contradicts boundedness of \(\widehat\sigma\).  Thus \(|A|\le1\), proving (i)\(\Rightarrow\)(iii).

Finally assume \(|A|\le1\).  The measure \(\nu\) in \eqref{eq:nu} has infinite support whenever \(c_-+c_+>0\), because it contains Lebesgue measure on a nondegenerate interval.  If \(B\ne\emptyset\), then \(S_B(v)>0\) for every \(v\ge0\), so the left tail also has infinite support.  A nonzero polynomial has only finitely many zeros, and therefore its square has strictly positive integral against any positive measure with infinite support.  This proves positive definiteness under \eqref{eq:pdcriterion}.  Conversely, if \(c_-=c_+=0\) and \(B=\emptyset\), then \(\nu\) is either the zero measure or the single atom \(\alpha\delta_0\), so a nonzero polynomial vanishing at \(0\) has zero squared norm.  Positive definiteness fails.  The criterion is proved.
\end{proof}

\begin{corollary}[Compactly supported positive members]\label{cor:attainable}
Within the normalized class \(x_0\in[a,b]\), a positive-definite member of the image has compact support if and only if \(B=\emptyset\).  In that case
\[
 d\nu(t)=\mathbf1_{[-c_-,c_+]}(t)\,dt+\alpha\delta_{c_+},
 \qquad \alpha\ge0,
\]
with \(c_-+c_+>0\).  Consequently no positive measure with a nonconstant absolutely continuous weight on a compact interval is obtained by the unweighted correspondence.  In particular, the non-Legendre Jacobi measures are excluded.  The Hermite and Laguerre measures are also excluded: every positive measure in \eqref{eq:nu} is bounded above by \(c_+\), whereas Hermite has two-sided unbounded support and Laguerre has support unbounded to the right.
\end{corollary}

\begin{proof}
If \(B=\emptyset\), the tail term in \eqref{eq:nu} vanishes and the support is contained in \([-c_-,c_+]\); positive definiteness forces \(c_-+c_+>0\).  If \(B\ne\emptyset\), then \(S_B(v)>0\) for all \(v\ge0\), so the support contains the whole half-line \(( -\infty,-c_-]\) and is not compact.  The explicit form of Theorem~\ref{thm:positivity} also proves the exclusion statement.  When the support is compact, the absolutely continuous part of every positive-definite member is Lebesgue measure with constant density one on a nondegenerate interval, possibly supplemented by one endpoint atom.  Therefore a compactly supported positive measure whose absolutely continuous density is nonconstant cannot occur.  In particular, a Jacobi density \((1-t)^\alpha(1+t)^\beta\) can belong to the image only in the Legendre case \(\alpha=\beta=0\), up to the affine normalization already built into the interval.  Hermite is impossible because its support is unbounded in both directions, while Laguerre is impossible because its support is unbounded to the right, whereas every measure in \eqref{eq:nu} is bounded above by \(c_+\).  These two exclusions use the uniqueness of the representing measure: by the exponential-moment estimate established in the proof of Theorem~\ref{thm:positivity}, every positive-definite member of the image is a determinate moment functional, so it cannot be represented both by \eqref{eq:nu} and by the Hermite or Laguerre measure.
\end{proof}

\begin{remark}[Laguerre--P\'olya type I recognition]\label{rem:recognition}
When \(B=\emptyset\), define
\[
 H_{c_-}(z):=z\mathcal L(e^{zt})+e^{-c_-z}.
\]
Theorem~\ref{thm:image} says that attainability is equivalent to the existence of \(c_-\ge0\) for which
\begin{equation}\label{eq:LP}
 H_{c_-}(z)=e^{c_+z}\prod_{\alpha\in A}(1+\alpha z),
 \qquad c_+\ge0,
 \quad \sum_{\alpha\in A}\alpha<\infty.
\end{equation}
Because \(H_{c_-}(0)=1\), this is exactly the normalized Laguerre--P\'olya class of type I with non-positive zeros.  In the language of total positivity, the Aissen--Schoenberg--Whitney--Edrei characterization identifies the corresponding entire generating functions with P\'olya-frequency sequences \cite{ASW,Edrei,Karlin}.  The point here is not to enlarge that classical class, but to recognize it as the entire part of the Stieltjes image.
\end{remark}

\begin{remark}[Centres outside the orthogonality interval]\label{rem:outside}
The restriction \(x_0\in[a,b]\) is essential to the four-parameter orientation used above.  For example, if \(x_0>b\), both endpoint values in \eqref{eq:mastergf} involve reciprocal jump products, so \(\mathcal L(e^{zt})\) becomes a difference of two left-type terms.  An analogous formula holds for \(x_0<a\).  We do not assert the positivity classification of Theorem~\ref{thm:positivity} in those configurations; they constitute a separate image problem.
\end{remark}

\section{Explicit realizations and the weighted lattice extension}\label{sec:dictionary}

The image theorem makes it possible to read many examples directly from the structural parameters.  We nevertheless give the principal realizations in detail.  The finite-jump proofs provide an intrinsic check of the general theorem, while the lattice examples belong to a different, weighted correspondence \((g,\varrho)\mapsto\mathcal L_{g,\varrho}\) and must not be conflated with the unweighted image theorem.

\subsection{Finitely many jumps to the right of the centre}

Let \(k\ge1\), \(a_1,\ldots,a_k>0\), and
\[
 0\le c_1<c_2<\cdots<c_k<1.
\]
On \([-1,1]\) consider
\begin{equation}\label{eq:gk}
 g(x)=x+\sum_{i=1}^k a_i\chi_{(c_i,\infty)}(x),
\end{equation}
with centre \(x_0=0\) and orthogonality interval \([-1,1)\).  Allowing \(c_1=0\) includes the model derivator \(g_1\) exactly, rather than only as a limiting position.  Denote by
\[
 e_j(a_1,\ldots,a_k)
 =\sum_{1\le i_1<\cdots<i_j\le k}a_{i_1}\cdots a_{i_j}
\]
the elementary symmetric functions, and put \(e_0=1\).

\begin{theorem}\label{thm:kjumps}
For the derivator \eqref{eq:gk}, \((g,0)\) is admissible and
\begin{equation}\label{eq:krein}
 \mathcal L(p)=\int_{-1}^1p(t)\,dt
 +\sum_{j=1}^k e_j(a_1,\ldots,a_k)p^{(j-1)}(1),
 \qquad p\in\mathbb R[t].
\end{equation}
Thus the functional depends on the jump sizes but not on the jump positions.  More precisely, the positions disappear and the sizes survive only through their elementary symmetric functions.
\end{theorem}

\begin{proof}
Admissibility is immediate from the left side of the centre.  For every \(x<0\), the oriented interval from \(0\) to \(x\) contains no jump of \(g\), and the continuous part of \(g\) is the identity.  Induction in the definition of the Stieltjes monomials therefore gives
\[
 g_{0,n}(x)=x^n,\qquad -1\le x<0.
\]
If \(\sum_{n=0}^N\lambda_ng_{0,n}=0\), restriction to \((-1,0)\) yields the ordinary polynomial identity
\[
 \sum_{n=0}^N\lambda_nx^n=0,
\]
so every coefficient \(\lambda_n\) vanishes.

We now compute the right endpoint explicitly.  Since the continuous part of \(g\) equals \(x\), Lemma~\ref{lem:egf} and \eqref{eq:stieltjes-exp} give, for \(x>c_k\),
\begin{equation}\label{eq:egf-right-k}
 \sum_{n\ge0}g_{0,n}(x)\frac{z^n}{n!}
 =e^{xz}\prod_{i=1}^k(1+a_i z)
 =e^{xz}\sum_{j=0}^k e_j(a_1,\ldots,a_k)z^j.
\end{equation}
Expanding \(e^{xz}\) and collecting the coefficient of \(z^n\) gives
\begin{align*}
 \frac{g_{0,n}(x)}{n!}
 &=\sum_{j=0}^{\min\{k,n\}}e_j(a_1,\ldots,a_k)
   \frac{x^{n-j}}{(n-j)!},
\end{align*}
and therefore
\begin{equation}\label{eq:gmon-k}
 g_{0,n}(x)=
 \sum_{j=0}^{\min\{k,n\}}e_j(a_1,\ldots,a_k)
 \frac{n!}{(n-j)!}x^{n-j},\qquad x>c_k.
\end{equation}
This coefficient calculation is also the precise reason the jump positions cancel: each crossing contributes only the multiplicative factor \(1+a_i z\) to the Stieltjes exponential, independently of where the crossing occurs.

At the endpoints,
\[
 g_{0,n+1}(-1)=(-1)^{n+1}
\]
and, by \eqref{eq:gmon-k},
\[
 g_{0,n+1}(1)=
 \sum_{j=0}^{\min\{k,n+1\}}e_j(a_1,\ldots,a_k)
 \frac{(n+1)!}{(n+1-j)!}.
\]
The master formula \eqref{eq:master} now gives
\begin{align*}
 m_n
 &=\frac{1-(-1)^{n+1}}{n+1}
 +\sum_{j=1}^{\min\{k,n+1\}}e_j(a_1,\ldots,a_k)
 \frac{n!}{(n+1-j)!}\\
 &=\int_{-1}^1t^n\,dt
 +\sum_{j=1}^k e_j(a_1,\ldots,a_k)
 \left.\frac{d^{j-1}}{dt^{j-1}}t^n\right|_{t=1},
\end{align*}
where terms with \(j-1>n\) vanish.  Equality on every monomial proves \eqref{eq:krein} by linearity.
\end{proof}

\begin{corollary}[One right jump: Uvarov]\label{cor:k1}
For a single jump of size \(\alpha>0\), placed anywhere in \([0,1)\),
\[
 \mathcal L(p)=\int_{-1}^1p(t)\,dt+\alpha p(1).
\]
The associated monic orthogonal polynomials are the endpoint Uvarov perturbation of the monic Legendre system,
\begin{equation}\label{eq:uvarov}
 \widehat q_n(t)=\widehat P_n(t)
 -\frac{\alpha\widehat P_n(1)}{1+\alpha K_{n-1}(1,1)}K_{n-1}(t,1),
 \qquad K_{n-1}(1,1)=\frac{n^2}{2}.
\end{equation}
\end{corollary}

\begin{proof}
The functional is Theorem~\ref{thm:kjumps} with \(k=1\).  Formula \eqref{eq:uvarov} is the classical one-mass Uvarov formula \cite{Uvarov1969}; the endpoint kernel identity follows from the standard Legendre Christoffel--Darboux kernel.
\end{proof}

\begin{corollary}[Two or more right jumps]\label{cor:k2}
If \(k\ge2\), \eqref{eq:krein} contains derivative evaluations at \(t=1\).  Thus it is an endpoint Dirac-derivative perturbation rather than an Uvarov point-mass perturbation.  In particular, by Theorem~\ref{thm:positivity}, it is not representable by a positive Borel measure.  The coefficients of the derivative evaluations are constrained by the Stieltjes geometry to be the elementary symmetric functions \(e_j(a_1,\ldots,a_k)\).
\end{corollary}

Functionals obtained by adding Dirac deltas and derivatives of Dirac deltas to a quasi-definite functional have been studied in a general setting; see, in particular, \cite{AlvarezNodarse2004}.  The specific contribution here is therefore not the abstract existence of such perturbations, but the classification of exactly how they arise from Stieltjes jumps, including the loss of all positional information and the symmetric-function constraint on their coefficients.

For \(k=2\), for example,
\[
 \mathcal L(p)=\int_{-1}^1p(t)\,dt+e_1p(1)+e_2p'(1),
 \qquad e_1=a_1+a_2,\quad e_2=a_1a_2.
\]
The first three Hankel determinants are
\[
 \widehat\Delta_1=e_1+2,
\]
\[
 \widehat\Delta_2=\frac13\bigl(8e_1-3e_2^2+12e_2+4\bigr),
\]
and
\[
 \widehat\Delta_3=\frac{16}{135}\bigl(9e_1-54e_2^2+36e_2+2\bigr).
\]
These finite determinants can remain positive for small \(e_2\), even though the full functional is not positive semidefinite.  This illustrates why finite Hankel tests cannot replace Theorem~\ref{thm:positivity}.

\subsection{A jump to the left: from a derivative expansion to a positive tail}

Let \(a>0\), \(-1<c<0\), and
\[
 g(x)=x+a\chi_{(c,\infty)}(x),
\]
again with centre \(0\) on \([-1,1)\).

\begin{theorem}\label{thm:left}
The pair \((g,0)\) is admissible.  Its Stieltjes monomials satisfy
\begin{equation}\label{eq:gmon-left}
 g_{0,n}(x)=
 \sum_{j=0}^n\frac{n!}{(n-j)!}(-a)^jx^{n-j},\qquad x\le c,
\end{equation}
and \(g_{0,n}(x)=x^n\) for \(x>c\).  The associated functional admits the two equivalent representations
\begin{equation}\label{eq:left-series}
 \mathcal L(p)=\int_{-1}^1p(t)\,dt
 +\sum_{j\ge1}(-1)^{j-1}a^jp^{(j-1)}(-1),
\end{equation}
where the sum is finite on every polynomial, and
\begin{equation}\label{eq:left-tail}
 {\displaystyle
 \mathcal L(p)=\int_{-1}^1p(t)\,dt
 +\int_{-\infty}^{-1}p(t)e^{(t+1)/a}\,dt.}
\end{equation}
In particular \(\mathcal L\) is positive definite, and neither representation depends on the jump position \(c\).
\end{theorem}

\begin{proof}
For \(x>c\), the oriented interval between the centre \(0\) and \(x\) does not cross the jump at \(c\).  Since the continuous part is the identity, induction in the monomial definition gives \(g_{0,n}(x)=x^n\).

For \(x\le c\), define the ordinary polynomial
\[
 q_n(x):=\sum_{j=0}^n\frac{n!}{(n-j)!}(-a)^jx^{n-j}.
\]
We verify that these are exactly the left restrictions of the Stieltjes monomials.  First,
\begin{align*}
 q_n'(x)
 &=\sum_{j=0}^{n-1}\frac{n!}{(n-j-1)!}(-a)^jx^{n-j-1}\\
 &=n\sum_{j=0}^{n-1}\frac{(n-1)!}{(n-1-j)!}(-a)^jx^{n-1-j}
 =nq_{n-1}(x).
\end{align*}
Thus the differential relation is correct on the continuous cell \(( -1,c)\).  At the jump, the Stieltjes integral requires the right value and the left value to satisfy
\begin{equation}\label{eq:jump-relation-left}
 g_{0,n}(c^+)-g_{0,n}(c)=na\,g_{0,n-1}(c).
\end{equation}
The polynomials \(q_n\) satisfy exactly this relation.  Indeed the finite geometric-operator identity
\[
 q_n=(I+aD)^{-1}x^n
 :=\sum_{j=0}^n(-aD)^j x^n
\]
is literal because \(D^{n+1}x^n=0\).  Hence
\[
 (I+aD)q_n=x^n,
\]
which, using \(q_n'=nq_{n-1}\), is equivalent at \(x=c\) to
\[
 c^n-q_n(c)=a q_n'(c)=na q_{n-1}(c).
\]
Since the right restriction is \(c^n\), this is \eqref{eq:jump-relation-left}.  Together with the common initial normalization at the centre, uniqueness of the recursive Stieltjes integral construction yields \(g_{0,n}(x)=q_n(x)\) on the left.  This proves \eqref{eq:gmon-left}.  Admissibility follows by restriction to \((c,1)\), where the monomials are the ordinary powers \(x^n\).

At the right endpoint, \(g_{0,n+1}(1)=1\).  At the left endpoint, \eqref{eq:gmon-left} gives
\[
 g_{0,n+1}(-1)=
 \sum_{j=0}^{n+1}\frac{(n+1)!}{(n+1-j)!}(-a)^j(-1)^{n+1-j}.
\]
Substitution in \eqref{eq:master} gives
\[
 m_n=\int_{-1}^1t^n\,dt
 +\sum_{j=1}^{n+1}(-1)^{j-1}a^j
 \left.\frac{d^{j-1}}{dt^{j-1}}t^n\right|_{t=-1},
\]
which is \eqref{eq:left-series} on monomials and hence on all polynomials.

We finally identify the derivative expansion with an ordinary tail integral.  Let \(p\) have degree \(d\).  Repeated integration by parts is legitimate because every polynomial times \(e^{(t+1)/a}\) tends to zero as \(t\to-\infty\).  One step gives
\[
 \int_{-\infty}^{-1}p(t)e^{(t+1)/a}\,dt
 =a p(-1)-a\int_{-\infty}^{-1}p'(t)e^{(t+1)/a}\,dt.
\]
Repeating this identity \(d\) more times and using \(p^{(d+1)}=0\) yields the exact finite formula
\[
 \int_{-\infty}^{-1}p(t)e^{(t+1)/a}\,dt
 =\sum_{r=0}^d(-1)^ra^{r+1}p^{(r)}(-1),
\]
which is precisely the second term of \eqref{eq:left-series} after setting \(j=r+1\).  Thus \eqref{eq:left-tail} holds.  Its measure has positive density on \(( -\infty,1)\), so every nonzero polynomial has strictly positive squared norm.  Positive definiteness follows.
\end{proof}

Theorem~\ref{thm:left} explains the orientation already visible in Theorem~\ref{thm:image}.  A right jump produces a polynomial factor \((1+az)\), hence a local endpoint operation.  A left jump produces its reciprocal, hence a resolvent-like integral operation and a smooth tail.  The earlier ``infinite derivative order'' description is therefore only a moment expansion; the underlying functional is an ordinary positive measure.

\subsection{Weighted unit-lattice derivators and factorial moments}

The following construction is a weighted extension and is not a special case of the unweighted map classified in Section~\ref{sec:image}.  Let
\begin{equation}\label{eq:g2}
 g_2(x)=\lceil x\rceil,\qquad x\ge0,
\end{equation}
with centre \(0\).  Its Stieltjes measure is counting measure on \(\mathbb N_0\).  The pair \((g_2,0)\) is admissible: by \eqref{eq:falling} below, a finite relation among the \(g\)-monomials restricts on \(\mathbb N_0\) to a relation among the falling-factorial polynomials \(s^{\underline n}\); these form a basis of \(\mathbb R[s]\), so all coefficients must vanish.  Thus admissibility concerns the derivator itself and is unaffected by whether a subsequently chosen weight has finite or infinite support.  At integer arguments
\begin{equation}\label{eq:falling}
 g_{0,n}(s)=s^{\underline n}:=s(s-1)\cdots(s-n+1).
\end{equation}
Let \(\varrho:\mathbb N_0\to[0,\infty)\) satisfy an exponential-moment condition
\begin{equation}\label{eq:rhoexp}
 \sum_{s=0}^\infty\varrho(s)R^s<\infty
 \quad\text{for some }R>1.
\end{equation}
This assumption is stronger than the existence of all ordinary moments and is imposed because we shall use the probability generating function analytically in a neighbourhood of \(w=1\).  It holds in each of the three examples below: the Poisson generating function is entire, the negative binomial one converges for \(|w|<1/q\) and \(1<1/q\), and the binomial one is a polynomial.  Define
\[
 \langle u,v\rangle_{S,\varrho}
 :=\sum_{s=0}^\infty (u\star v)(s)\varrho(s),
\qquad
 \mathcal L_{\varrho}(t^n):=\sum_{s=0}^\infty s^{\underline n}\varrho(s).
\]

\begin{theorem}[Factorial-moment transfer]\label{thm:discrete}
Under \eqref{eq:rhoexp}, the moments of \(\mathcal L_\varrho\) are the factorial moments of \(\varrho\),
\begin{equation}\label{eq:factmom}
 m_n=\mathcal L_\varrho(t^n)
 =\sum_{s=0}^\infty s^{\underline n}\varrho(s).
\end{equation}
If
\[
 \Phi_\varrho(w):=\sum_{s=0}^\infty\varrho(s)w^s,
\]
then
\begin{equation}\label{eq:pgf}
 \mathcal L_\varrho(e^{zt})=\Phi_\varrho(1+z)
\end{equation}
whenever \(|1+z|<R\), and in particular in a neighbourhood of \(z=0\).
\end{theorem}

\begin{proof}
Formula \eqref{eq:factmom} follows directly from \eqref{eq:falling}.  For the generating function, absolute convergence under \eqref{eq:rhoexp} permits interchange of the two sums for \(z\) near zero:
\begin{align*}
 \mathcal L_\varrho(e^{zt})
 &=\sum_{n=0}^\infty\frac{z^n}{n!}
   \sum_{s=0}^\infty s^{\underline n}\varrho(s)\\
 &=\sum_{s=0}^\infty\varrho(s)
   \sum_{n=0}^s\binom{s}{n}z^n
 =\sum_{s=0}^\infty\varrho(s)(1+z)^s
 =\Phi_\varrho(1+z).
\end{align*}
The inner sum is finite because \(s^{\underline n}=0\) for \(n>s\).
\end{proof}

\begin{example}[Meixner weight: scaled Laguerre]\label{ex:meixner}
Let
\[
 \varrho(s)=(1-q)^\beta\frac{(\beta)_s}{s!}q^s,
 \qquad \beta>0,\quad0<q<1.
\]
Then
\[
 \Phi_\varrho(w)=\left(\frac{1-q}{1-qw}\right)^\beta,
\]
so, with \(\lambda=q/(1-q)\),
\[
 \mathcal L_\varrho(e^{zt})=(1-\lambda z)^{-\beta}.
\]
This is the moment generating function of the Gamma density
\[
 \frac{1}{\lambda^\beta\Gamma(\beta)}t^{\beta-1}e^{-t/\lambda}\mathbf1_{(0,\infty)}(t)\,dt.
\]
Hence the transported monic orthogonal polynomials are
\begin{equation}\label{eq:scaled-laguerre}
 q_n(t)=(-1)^n n!\lambda^n L_n^{(\beta-1)}(t/\lambda),
\end{equation}
and the corresponding \(g\)-polynomials are \(\Phi(q_n)\).  The scale factor in \eqref{eq:scaled-laguerre} is essential unless \(\lambda=1\).
\end{example}

\begin{example}[Charlier weight: complete degeneracy]\label{ex:charlier}
Let
\[
 \varrho(s)=e^{-\mu}\frac{\mu^s}{s!},\qquad \mu>0.
\]
Then \(\Phi_\varrho(w)=e^{\mu(w-1)}\) and
\[
 \mathcal L_\varrho(e^{zt})=e^{\mu z}.
\]
Therefore \(\mathcal L_\varrho(p)=p(\mu)\).  Its Hankel matrix has rank one: \(\widehat\Delta_1=1\) and \(\widehat\Delta_n=0\) for every \(n\ge2\).  Thus no infinite sequence with nonzero orthogonality constants exists for the \(\star\)-orthogonality, despite the fact that the original pointwise-product discrete weight is the classical Charlier weight.
\end{example}

\begin{example}[Krawtchouk weight: a finite-order derivative functional]\label{ex:krawtchouk}
Let
\[
 \varrho(s)=\binom Ns p^s(1-p)^{N-s},
 \qquad s=0,\ldots,N,\quad0<p<1.
\]
The probability generating function is \(\Phi_\varrho(w)=(1-p+pw)^N\), so
\begin{equation}\label{eq:krawmgf}
 \mathcal L_\varrho(e^{zt})=(1+pz)^N
\end{equation}
and
\begin{equation}\label{eq:krawmom}
 m_n=N^{\underline n}p^n,
 \qquad m_n=0\quad(n>N).
\end{equation}
The functional is therefore
\begin{equation}\label{eq:krawfunctional}
 {\displaystyle
 \mathcal L_\varrho(q)=
 \sum_{j=0}^N\binom Nj p^j q^{(j)}(0).}
\end{equation}
In particular it is not the moment functional of the original positive Krawtchouk measure.  Already
\[
 \mathcal L_\varrho\bigl((t-Np)^2\bigr)
 =m_2-m_1^2=-Np^2<0.
\]

For completeness, the Hankel determinants can be computed in closed form.  For an \(n\times n\) Hankel matrix \((m_{i+j})_{i,j=0}^{n-1}\), factor \(p^i\) from row \(i\) and \(p^j\) from column \(j\) to obtain
\[
 \widehat\Delta_n
 =p^{n(n-1)}\det\bigl(N^{\underline{i+j}}\bigr)_{i,j=0}^{n-1}.
\]
Since \(N^{\underline k}=(-1)^k(-N)_k\), multiplying row \(i\) and column \(j\) by the corresponding signs shows that the signs cancel in the determinant.  It remains to evaluate
\[
 D_n(a):=\det\bigl((a)_{i+j}\bigr)_{i,j=0}^{n-1}
\]
with \(a=-N\).  The factorization \((a)_{i+j}=(a)_i(a+i)_j\) gives
\[
 D_n(a)=\left(\prod_{i=0}^{n-1}(a)_i\right)
 \det\bigl((a+i)_j\bigr)_{i,j=0}^{n-1}.
\]
For each fixed \(j\), the function \(i\mapsto(a+i)_j\) is a monic polynomial in \(i\) of degree \(j\).  Replacing the monomial basis \(1,i,\ldots,i^{n-1}\) by these monic polynomials is a unit upper-triangular change of basis, so the last determinant equals the Vandermonde determinant at the nodes \(0,1,\ldots,n-1\):
\[
 \det\bigl((a+i)_j\bigr)_{i,j=0}^{n-1}
 =\prod_{0\le i<k\le n-1}(k-i)
 =\prod_{j=0}^{n-1}j!.
\]
Consequently
\begin{equation}\label{eq:krawdet}
 {\displaystyle
 \widehat\Delta_n
 =p^{n(n-1)}\prod_{j=0}^{n-1}j!\,(-N)_j.}
\end{equation}
Consequently \(\widehat\Delta_n\ne0\) for \(1\le n\le N+1\), while \(\widehat\Delta_{N+2}=0\).  Thus monic orthogonal polynomials with nonzero orthogonality constants exist through degree \(N\); the next monic polynomial, although algebraically constructible from the nonsingular \((N+1)\times(N+1)\) moment block and orthogonal to every polynomial of degree at most \(N\), is isotropic: its self-pairing is zero because \(h_{N+1}=\widehat\Delta_{N+2}/\widehat\Delta_{N+1}=0\).
\end{example}

\begin{remark}[Continuous weighted case]
If \(g\) is continuous and strictly increasing, then \(u=g(x)-g(x_0)\) is an ordinary coordinate and \(\Phi(p)=p(u)\).  A nonnegative function weight \(\varrho\) therefore produces the usual absolutely continuous weighted moment functional in the variable \(u\), obtained by pushing \(\varrho\,d\mu_g\) forward under \(x\mapsto u\).  This recovers ordinary weighted polynomial theory in the new coordinate; it does \emph{not} imply that every abstract moment functional is obtained from a function weight.
\end{remark}

\section{The model derivator $g_{1}$}\label{sec:model}

In this section we analyse in detail the $\star$-orthogonal $g$-polynomials for the model derivator defined in~\eqref{eq:ex1}, both because it is the simplest nontrivial case and because it illustrates every item of Corollary~\ref{cor:transfer}. By Corollary~\ref{cor:k1}, the corresponding moment functional is the Legendre functional modified by a mass $\alpha$ at $t=1$, and the family below is the $\Phi$-image of the Uvarov family \eqref{eq:uvarov}; we nevertheless present the computation intrinsically, in the $g$-monomial basis and with the Stieltjes derivative, since it is that realisation, and not the abstract family, which is specific to the present framework. We shall introduce first a second-order $g$-differential equation having $g$-polynomial solutions, and later derive that these polynomials are indeed $g$-orthogonal.

\subsection{The $\star$-orthogonal $g$-polynomials of $g_{1}$}\label{sub:model-main}

Let $g=g_{1}$ be the derivator introduced in \eqref{eq:ex1}. Notice that the $g$-monomials explicitly given in \eqref{eq:explicitmonomials} are no longer symmetric, unlike in the continuous case for the basis $\{x^{n}\}_{n}$ so that many properties of symmetry of Legendre polynomials will not be extended to this new situation.

Let us consider the following second order $g$-differential equation
\begin{equation}\label{eq:theequation}
\sigma_{3}(x;n) \star (P_{n}(x;g))''_{g} + \tau_{2}(x;n) \star (P_{n}(x;g))'_{g} + \varpi_{1}(x;n)  \star P_{n}(x;g)=0,
\end{equation}
where $\sigma_{3}(x;n)$ is a $g$-polynomial of degree $3$, $\tau_{2}(x;n)$ is a $g$-polynomial of degree $2$, and $\varpi_{1}(x;n)$ is a $g$-polynomial of degree $1$:
\begin{equation}\label{eq:thepolynomials}
\begin{cases}
\sigma_{3}(x;n)=&\left(\alpha  n^2+2\right) \left(\alpha  (n+1)^2+2\right) (g_{0,3}(x)-g_{0,1}(x)) \\
&+ (-\alpha  n (n+1) (\alpha  n (n+1)+4)+2 \alpha -4) g_{0,2}(x) \\
&+ (\alpha  (n (n+1) (\alpha  n (n+1)+4)-2)+4) g_{0,0}(x), \\[3mm]
\tau_{2}(x;n)=&(2 \left(\alpha  n^2+2\right) \left(\alpha  (n+1)^2+2\right)) g_{0,2}(x) \\
&+(-2 (\alpha  (n (n+1) (\alpha  n (n+1)+4)-4)+4))g_{0,1}(x) + 4 \alpha g_{0,0}(x) {,}\\[3mm]
\varpi_{1}(x;n)=&n (n+1) \left(\left(\alpha  \left(n (n+1) \left(\alpha
   \left(n^2+n-2\right)+4\right)-6\right)+4\right)g_{0,0}(x) \right. \\
   & \left.-\left(\alpha
   n^2+2\right) \left(\alpha  (n+1)^2+2\right) g_{0,1}(x)\right).
\end{cases}
\end{equation}
Let us assume that the above equation has a $g$-polynomial solution that can be expanded in terms of the $g$-monomials  \eqref{eq:explicitmonomials} as in \eqref{eq:expansion}. By using \eqref{eq:dgm} and
\begin{equation}
(u(x) \star v(x))'_{g}=(u(x))'_{g} \star v(x) + u(x) \star (v(x))'_{g},
\end{equation}
{ Indeed, this identity is first checked on the basis of $g$-monomials, since
\[
(g_{0,i}\star g_{0,j})'_g=(g_{0,i+j})'_g=(i+j)g_{0,i+j-1}
\]
and
\[
(g_{0,i})'_g\star g_{0,j}+g_{0,i}\star (g_{0,j})'_g
=i g_{0,i-1}\star g_{0,j}+j g_{0,i}\star g_{0,j-1}=(i+j)g_{0,i+j-1},
\]
and then extended by bilinearity.}
{
We obtain, after collecting coefficients of each $g_{0,j}$, two independent recurrence schemes for the coefficients $c_{n,j}$ in the expansion, after substituting the expansion
\[
P_n(x;g)=\sum_{j=0}^n c_{n,j}g_{0,j}(x)
\]
into the differential equation and using
\[
(g_{0,j})'_g=jg_{0,j-1},
\qquad
(g_{0,j})''_g=j(j-1)g_{0,j-2},
\qquad
 g_{0,i}\star g_{0,j}=g_{0,i+j},
\]
one obtains an explicit finite recurrence for the coefficients. More conveniently,
the resulting solution can be organized into two families according to the parity of
$n-j$:
\[
c_{n,n-2r}=U_{r,n},
\qquad
c_{n,n-(2r+1)}=T_{r,n}.
\]
These two families satisfy the first-order recurrence relations
\[
U_{0,n}=1,
\]
\[
U_{r,n}
=
-\frac{\alpha(n^2+2r)+2}{\alpha(n^2+2r-2)+2}
\frac{(n-2r+2)(n-2r+1)}{2r(2n-2r+1)}U_{r-1,n},
\qquad r\geq1,
\]
and
\[
T_{0,n}=-\frac{\alpha n}{\alpha n^2+2},
\]
\[
T_{r,n}
=
-
\frac{(n-2r+1)(n-2r)}{2r(2n-2r+1)}T_{r-1,n},
\qquad r\geq1.
\]
{Solving these recurrences yields}
}
\begin{equation}\label{eq:explg1}
{
\hat{P}_{n}(x;g)
=
 g_{0,n}(x)
 +\sum_{r=1}^{\lfloor n/2\rfloor}U_{r,n}g_{0,n-2r}(x)
 +\sum_{r=0}^{\lfloor (n-1)/2\rfloor}T_{r,n}g_{0,n-2r-1}(x),
}
\end{equation}
{where}
\begin{equation}\label{eq:Uclosed-paper}
{
U_{r,n}
=
(-1)^r
\frac{\bigl(\alpha(n^2+2r)+2\bigr)n!}
{2^r r!(\alpha n^2+2)(n-2r)!\displaystyle\prod_{s=0}^{r-1}(2n-2s-1)},
\qquad 1\le r\le \left\lfloor\frac n2\right\rfloor,
}
\end{equation}
\begin{equation}\label{eq:Tclosed-paper}
{
T_{r,n}
=
(-1)^{r+1}
\frac{\alpha n!}
{2^r r!(\alpha n^2+2)(n-2r-1)!\displaystyle\prod_{s=0}^{r-1}(2n-2s-1)},
\qquad 0\le r\le \left\lfloor\frac{n-1}{2}\right\rfloor,
}
\end{equation}
{and empty products are understood to be equal to $1$.}
\begin{remark}
{From \eqref{eq:Tclosed-paper} we have}
\begin{equation}
{\lim_{\alpha \to 0} T_{r,n}=0.}
\end{equation}
{Moreover, from \eqref{eq:Uclosed-paper},}
\begin{equation}
{\lim_{\alpha \to 0} U_{r,n}=(-1)^r \frac{n!}{2^r r!(n-2r)!\displaystyle\prod_{s=0}^{r-1}(2n-2s-1)}, \quad r\geq1,}
\end{equation}
{which coincides with the even coefficients of the $n$-th degree monic Legendre polynomial. Since, for the $g$-monomials,}
\begin{equation}
{\lim_{\alpha \to 0} g_{0,n}(x)=x^{n},}
\end{equation}
{we have}
\begin{equation}
{\lim_{\alpha \to 0} \hat{P}_{n}(x;g)=\hat{P}_{n}(x),}
\end{equation}
{where $\hat{P}_{n}(x)$ stands for the monic Legendre polynomial of degree $n$.}
\end{remark}

{Let us now introduce the moments}
\begin{equation}\label{eq:momentsS-def}
{v_{i,j}=\langle g_{0,i},g_{0,j} \rangle_{S}=
\int_{[-1,1)}\left(g_{0,i}\star g_{0,j}\right)d\mu_g
=\int_{[-1,1)}g_{0,i+j}\,d\mu_g.}
\end{equation}
{Using the explicit expression of the $g$-monomials and of the derivator $g=g_1$, these moments are computed as follows:}
\begin{equation}\label{eq:mom10}
{v_{0,0}=\langle g_{0,0},g_{0,0} \rangle_{S}=\int_{[-1,1)}1\,d\mu_g=2+\alpha,}
\end{equation}
{and, for $i+j>0$,}
\begin{multline}\label{eq:mom1n}
{v_{i,j}=\langle g_{0,i},g_{0,j} \rangle_{S}
=\int_{-1}^{0}s^{i+j}\,ds+
\int_{0}^{1}\left(s^{i+j}+\alpha(i+j)s^{i+j-1}\right)\,ds}\\
{=\frac{(-1)^{i+j}+1}{i+j+1}+\alpha
=\frac{\alpha(i+j+1)+(-1)^{i+j}+1}{i+j+1}.}
\end{multline}

{
\begin{lemma}
Let \(g=g_1\) be the derivator defined in \eqref{eq:ex1}, with
\(\alpha>0\). The bilinear form
\[
\langle u,v\rangle_S
=
\int_{[-1,1)}(u\star v)\,d\mu_g
\]
is symmetric and nondegenerate on \(\mathbb P_g\).

Moreover, if
\[
v_{i,j}:=
\langle g_{0,i},g_{0,j}\rangle_S,
\qquad i,j\geq0,
\]
and
\[
\widehat\Delta_n
:=
\det\left(v_{i,j}\right)_{i,j=0}^{n-1},
\qquad n\geq1,
\]
then
\[
\widehat\Delta_1=\alpha+2
\]
and, for every \(n\geq1\),
\[
\widehat\Delta_{n+1}
=
\widehat\Delta_n\,
\frac{2^{2n+1}(n!)^4}{(2n+1)((2n)!)^2}
\,
\frac{\alpha(n+1)^2+2}{\alpha n^2+2}.
\]
Equivalently, the following closed formula holds:
\[
\widehat\Delta_n
=
2^{n^2-1}(\alpha n^2+2)
\prod_{k=0}^{n-1}
\frac{(k!)^4}{(2k+1)((2k)!)^2},
\qquad n\geq1.
\]
In particular,
\[
\widehat\Delta_n>0,
\qquad n\geq1.
\]
\end{lemma}

\begin{proof}
The symmetry follows from the commutativity of the \(\star\)-product. Indeed,
for the basis of \(g\)-monomials we have
\[
g_{0,i}\star g_{0,j}
=
g_{0,i+j}
=
g_{0,j}\star g_{0,i}.
\]
By bilinearity, this implies
\[
u\star v=v\star u
\]
for all \(u,v\in\mathbb P_g\). Hence
\[
\langle u,v\rangle_S
=
\int_{[-1,1)}(u\star v)\,d\mu_g
=
\int_{[-1,1)}(v\star u)\,d\mu_g
=
\langle v,u\rangle_S.
\]

{For the derivator \(g_1\), the moments are given by \eqref{eq:mom10} and \eqref{eq:mom1n}. Let $H_{n}$} defined in \eqref{eq:hankelc} be the classical Legendre Hankel matrix and let
\[
\mathbf 1_n=(1,\ldots,1)^T\in\mathbb R^n.
\]
The previous expression for the moments gives
\[
\left(v_{i,j}\right)_{i,j=0}^{n-1}
=
H_n+\alpha \mathbf 1_n\mathbf 1_n^T.
\]
By the matrix determinant lemma,
\[
\widehat\Delta_n
=
\det\left(H_n+\alpha \mathbf 1_n\mathbf 1_n^T\right)
=
\det H_n
\left(
1+\alpha\,\mathbf 1_n^T H_n^{-1}\mathbf 1_n
\right).
\]

The vector \(\mathbf 1_n\) is the vector of the monomial basis evaluated at
\(x=1\). Hence
\[
\mathbf 1_n^T H_n^{-1}\mathbf 1_n
=
K_{n-1}(1,1),
\]
where \(K_{n-1}\) is the Christoffel-Darboux kernel associated with the monic
Legendre polynomials. If \(\{\widehat P_j\}_{j\geq0}\) denotes the sequence of
monic Legendre polynomials, then
\[
K_{n-1}(1,1)
=
\sum_{j=0}^{n-1}
\frac{\widehat P_j(1)^2}{\|\widehat P_j\|^2}.
\]
Using the classical normalization of monic Legendre polynomials, one has
\[
\frac{\widehat P_j(1)^2}{\|\widehat P_j\|^2}
=
\frac{2j+1}{2}.
\]
Therefore,
\[
K_{n-1}(1,1)
=
\sum_{j=0}^{n-1}\frac{2j+1}{2}
=
\frac{n^2}{2}.
\]
Thus
\[
\widehat\Delta_n
=
\det H_n
\left(
1+\frac{\alpha n^2}{2}
\right).
\]

Now we use the classical expression for the Legendre Hankel determinant \eqref{eq:dethankelc} to obtain
\[
\widehat\Delta_n
=
\left(
\prod_{k=0}^{n-1}
\frac{2^{2k+1}(k!)^4}{(2k+1)((2k)!)^2}
\right)
\left(
1+\frac{\alpha n^2}{2}
\right),
\]
which can be rewritten as
\[
\widehat\Delta_n
=
2^{n^2}
\left(
1+\frac{\alpha n^2}{2}
\right)
\prod_{k=0}^{n-1}
\frac{(k!)^4}{(2k+1)((2k)!)^2}
=
2^{n^2-1}(\alpha n^2+2)
\prod_{k=0}^{n-1}
\frac{(k!)^4}{(2k+1)((2k)!)^2}.
\]
This proves the closed formula.

Let us now check that the closed formula is equivalent to the recurrence stated in
the lemma. For \(n=1\), the closed formula gives
\[
\widehat\Delta_1
=
2^{0}(\alpha+2)
\frac{(0!)^4}{1(0!)^2}
=
\alpha+2.
\]
Moreover, for \(n\geq1\),
\[
\frac{\widehat\Delta_{n+1}}{\widehat\Delta_n}
=
\frac{
2^{(n+1)^2-1}(\alpha(n+1)^2+2)
\displaystyle\prod_{k=0}^{n}
\frac{(k!)^4}{(2k+1)((2k)!)^2}
}{
2^{n^2-1}(\alpha n^2+2)
\displaystyle\prod_{k=0}^{n-1}
\frac{(k!)^4}{(2k+1)((2k)!)^2}
}.
\]
Therefore,
\[
\frac{\widehat\Delta_{n+1}}{\widehat\Delta_n}
=
2^{(n+1)^2-n^2}
\frac{\alpha(n+1)^2+2}{\alpha n^2+2}
\frac{(n!)^4}{(2n+1)((2n)!)^2}.
\]
Since
\[
(n+1)^2-n^2=2n+1,
\]
we obtain
\[
\frac{\widehat\Delta_{n+1}}{\widehat\Delta_n}
=
\frac{2^{2n+1}(n!)^4}{(2n+1)((2n)!)^2}
\,
\frac{\alpha(n+1)^2+2}{\alpha n^2+2}.
\]
Equivalently,
\[
\widehat\Delta_{n+1}
=
\widehat\Delta_n\,
\frac{2^{2n+1}(n!)^4}{(2n+1)((2n)!)^2}
\,
\frac{\alpha(n+1)^2+2}{\alpha n^2+2}.
\]
Thus the recurrence and the closed formula are equivalent.

Finally, since \(\alpha>0\), all the factors in the closed formula are strictly
positive. Hence
\[
\widehat\Delta_n>0,
\qquad n\geq1.
\]
In particular, all the moment matrices
\[
\left(v_{i,j}\right)_{i,j=0}^{n-1}
\]
are nonsingular. Therefore the bilinear form
\(\langle\cdot,\cdot\rangle_S\) is nondegenerate on \(\mathbb P_g\).
\end{proof}
}

By combining the explicit expression \eqref{eq:explg1} for the monic $g$-polynomial solution of \eqref{eq:theequation} as well as the latter explicit expressions for the moments {\eqref{eq:mom10} and \eqref{eq:mom1n}}, we obtain that
\begin{multline}
\langle \hat{P}_{n}(x;g),\hat{P}_{m}(x;g) \rangle_{S} = \int_{-1}^{1} (\hat{P}_{n}(s;g) \star \hat{P}_{m}(s;g)) d\mu_{g}(s) \\
= {\int_{-1}^{1} (\hat{P}_{n}(s;g) \star \hat{P}_{m}(s;g)) ds + \alpha (\hat{P}_{n}(\cdot;g) \star \hat{P}_{m}(\cdot;g))(0)} ={d_{n}^{2} \delta_{n,m}},
\end{multline}
{where
\begin{equation*}
d_{n}^{2}=\begin{cases}
\displaystyle {\alpha +2}, & n=0, \\[3mm]
\displaystyle  \frac{2^{2 n+1} \left(\alpha  (n+1)^2+2\right)}{(2 n+1) \left(\alpha  n^2+2\right) \binom{2 n}{n}^2}, & n \geq 1.
\end{cases}
\end{equation*}
}
As a consequence, we have obtained that the sequence of monic $g$-polynomials given explicitly in \eqref{eq:explg1} are indeed $g$-orthogonal $g$-polynomials. We shall refer to them as monic $g$-Legendre $g$-polynomials.

Let us now expand
\begin{equation}
{g_{0,1}(x) \star \hat{P}_{n}(x;g)=\hat{P}_{n+1}(x;g)+\sum_{j=0}^{n} c_{j,n} \hat{P}_{j}(x;g),}\quad {\hat P_{-1}(x;g):=0.}
\end{equation}
The Fourier coefficients can be computed as
\begin{equation}
c_{j,n}=\frac{\langle g_{0,1}(x) \star \hat{P}_{n}(x;g), \hat{P}_{j}(x;g) \rangle_{S}}{\langle \hat{P}_{j}(x;g), \hat{P}_{j}(x;g) \rangle_{S}}.
\end{equation}
Next, we use the symmetry property which can be easily proved
\begin{equation}
c_{j,n}=\frac{\langle g_{0,1}(x) \star \hat{P}_{n}(x;g), \hat{P}_{j}(x;g) \rangle_{S}}{\langle \hat{P}_{j}(x;g),\hat{P}_{j}(x;g) \rangle_{S}} =\frac{\langle \hat{P}_{n}(x;g), g_{0,1}(x) \star \hat{P}_{j}(x;g) \rangle_S}{\langle \hat{P}_{j}(x;g),\hat{P}_{j}(x;g) \rangle_{S}}.
\end{equation}
Since $g_{0,1}(x) \star \hat{P}_{j}(x;g) $ is a $g$-polynomial of degree at most $j+1$, by orthogonality all the coefficients from $j=0$ up to $n-2$ are zero. Therefore, by using the explicit expression of the $g$-polynomials given in \eqref{eq:explg1}, the $g$-differential equation \eqref{eq:theequation}, {after direct coefficient comparison}{,} we find that the monic Legendre $g$-orthogonal $g$-polynomials satisfy
\begin{equation}
g_{0,1}(x) \star \hat{P}_{n}(x;g)=\hat{P}_{n+1}(x;g)+ \beta_{n} \hat{P}_{n}(x;g) + \gamma_{n} \hat{P}_{n-1}(x;g),
\end{equation}
where the coefficients are explicitly given by
\begin{align}
\label{eq:betan}
\beta_{n}&=\frac{\alpha\left(2-\alpha n(n+1)\right)}{\left(\alpha n^{2}+2\right)\left(\alpha (n+1)^{2}+2\right)}, \\
\label{eq:gamman}
\gamma_{n}&=\frac{n^2 \left(\alpha  (n-1)^2+2\right) \left(\alpha  (n+1)^2+2\right)}{\left(4 n^2-1\right) \left(\alpha  n^2+2\right)^2},
\end{align}
(see Figures~\ref{figure2},~\ref{figure3} and \ref{figure4}). Both coefficients are, as Corollary~\ref{cor:transfer}(iii) predicts, the recurrence coefficients of the Uvarov--Legendre family \eqref{eq:uvarov}, and the denominators are the Uvarov factors $2\bigl(1+\alpha K_{n-1}(1,1)\bigr)=\alpha n^{2}+2$. We would like to emphasize that
\begin{equation}
\lim_{\alpha \to 0} \beta_{n}=0, \quad \lim_{\alpha \to 0} \gamma_{n}=\frac{n^2}{4 n^2-1},
\end{equation}
i.e. the coefficients of the {three-term} recurrence relation satisfied by classical monic Legendre polynomials \cite[Eq. (9.8.65)]{KOEKOEK}.

\begin{remark}
It is interesting to compute the limit of the polynomials \eqref{eq:thepolynomials} as $\alpha \to 0$. In that situation, $g_{0,n}(x) \to x^{n}$ and then
\begin{equation}
\begin{cases}
\displaystyle \lim_{\alpha \to 0} \sigma_{3}(x;n)=4 (-1 + x)^2 (1 + x), \\[3mm]
\displaystyle \lim_{\alpha \to 0} \tau_{2}(x;n)=8 (-1 + x) x, \\[3mm]
\displaystyle \lim_{\alpha \to 0} \varpi_{1}(x;n)=-4 n (1 + n) (-1 + x).
\end{cases}
\end{equation}
If we simplify the common factor $(x-1)$ and divide by $4$, we obtain the differential equation of Legendre polynomials.

\end{remark}

\begin{remark}
Classical continuous (discrete or their $q$-analogues, in linear or nonuniform lattices) orthogonal polynomials satisfy a {second-order} differential (difference, $q$-difference or on nonuniform lattices \cite{NSU}) of hypergeometric type. In the classical continuous situation the equation reads as
\begin{equation}
\sigma(x) y''(x)+\tau(x) y'(x)+\lambda_{n} y(x)=0,
\end{equation}
where $\sigma$ and $\tau$ are polynomials of at most second and first degree, respectively, and $\lambda_{n}$ is a constant. The hypergeometric character of the equation means that the derivative of any solution of the equation $v_{1}(x)=y'(x)$ satisfies an equation of the same type. More precisely,
\begin{equation}
\sigma(x) v_{1}''(x)+(\tau(x)+\sigma'(x)) v_{1}'(x)+(\lambda_{n}+\tau') v_{1}(x)=0.
\end{equation}

Some authors (see{, e.g.,} \cite{zarzo}) have studied the case in which $\tau(x)$ depends also on $n$, giving rise to varying weights.

When dealing with $g$-orthogonal $g$-polynomials the situation is more complicated. The $g$-orthogonal $g$-polynomials satisfy a second order $g$-differential equation, but the polynomials $\sigma$ and $\tau$ are of higher degree (so that, the constant multiplying $y(x)$ becomes a polynomial) and the coefficients depend on $n$. The equation given in \eqref{eq:theequation} is {no} longer of hypergeometric type.
\end{remark}

\begin{figure}[!ht]
\centering
\includegraphics[scale=0.9]{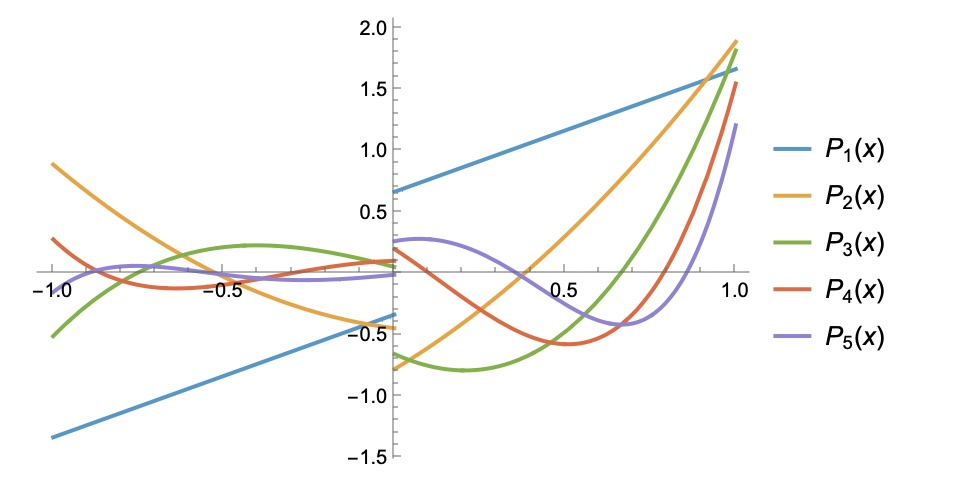}
\caption{The first five monic orthogonal $g$-Legendre $g$-polynomials $\hat{P}_{n}(x;g)$ for $\alpha=1$.}
\label{figure2}
\end{figure}

\begin{figure}[!ht]
\centering
\includegraphics[scale=0.4]{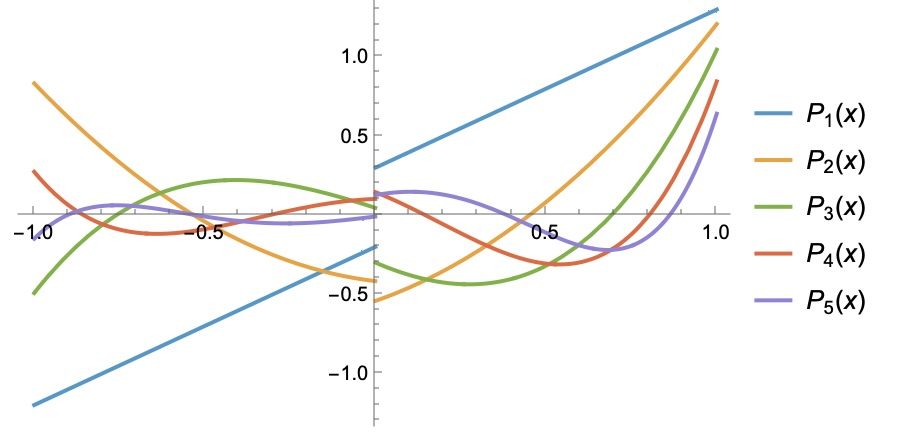}
\caption{The first five monic orthogonal $g$-Legendre $g$-polynomials $\hat{P}_{n}(x;g)$ for $\alpha=1/2$.}
\label{figure3}
\end{figure}

\begin{figure}[!ht]
\centering
\includegraphics[scale=0.4]{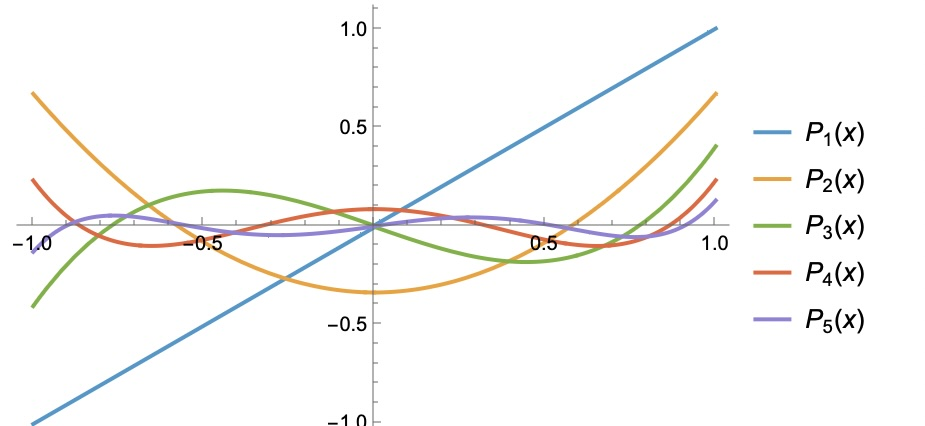}
\caption{The first five monic orthogonal $g$-Legendre $g$-polynomials $\hat{P}_{n}(x;g)$ for $\alpha=1/100$.}
\label{figure4}
\end{figure}

\subsection{Comparison of the two orthogonalities}\label{sub:comparison}

It is important to compare precisely the two notions of orthogonality considered in this paper, for the model derivator $g_{1}$.

In Section~\ref{sec:gop}, orthogonality is defined through the ordinary pointwise product of functions and the Stieltjes measure $\mu_{g}$. For $g_{1}$ this is the Legendre measure perturbed by a Dirac mass of size $\alpha$ at the \emph{interior} point $x=0$, that is, a classical Uvarov modification at an interior point.

In the present section, orthogonality is defined through the intrinsic product $\star$. Since $\star$ is not the pointwise product, the two bilinear forms are different, and so are the two families: compare $\hat P_{2}^{g}(x)=x^{2}-2/(3(\alpha+2))$ with $\hat P_{2}(x;g)$ obtained from \eqref{eq:explg1}. Nevertheless, by Corollary~\ref{cor:k1}, the second family is again a Uvarov--Legendre family, now with the mass placed at the \emph{endpoint} $t=1$ of the auxiliary variable, and written in the $g$-monomial basis through $\Phi$. Explicitly,
\begin{equation}\label{eq:identification}
\hat P_{n}(x;g)=\Phi(\hat q_{n})(x),\qquad
\hat q_{n}(t)=\hat P_{n}(t)-\frac{\alpha\hat P_{n}(1)}{1+\alpha K_{n-1}(1,1)}K_{n-1}(t,1),
\end{equation}
and every formula of this section (the coefficients \eqref{eq:Uclosed-paper} and \eqref{eq:Tclosed-paper}, the norms $d_{n}^{2}$, the recurrence coefficients \eqref{eq:betan}--\eqref{eq:gamman} and the $g$-differential equation \eqref{eq:theequation}) is the $\Phi$-image of the corresponding Uvarov formula. In particular, the identity \eqref{eq:identification} provides a second, structural proof of the orthogonality relation established above.

We therefore do not claim that the $\star$-orthogonality of Section~\ref{sec:gogp} produces a new abstract orthogonal family; by Theorem~\ref{thm:transfer-alg} and Remark~\ref{rem:honest} it cannot. What the model derivator illustrates is the mechanism: a single jump to the right of the centre transports a Dirac mass to an endpoint of the auxiliary variable, and the resulting ordinary polynomial family acquires an explicit realisation as a $g$-polynomial function of $x$, satisfying an explicit second-order $g$-differential equation with respect to $D_{g}$. The situation changes qualitatively as soon as the derivator has two or more right-hand jumps: Corollary~\ref{cor:k2} shows that derivative evaluations appear and Theorem~\ref{thm:positivity} rules out representation by a positive measure.

\section{Conclusions and open problems}\label{sec:conclusions}

The main conclusion of this work is a classification rather than the introduction of a new abstract polynomial algebra.  Once the Stieltjes monomials are linearly independent, the \(\star\)-algebra is isomorphic to \(\mathbb R[t]\) through \(\Phi(t^n)=g_{x_0,n}\), and \(D_g\) is transported to the ordinary derivative.  Therefore every \(\star\)-orthogonal family is the image of an ordinary monic orthogonal polynomial system associated with a quasi-definite linear moment functional.  The one-jump family analysed in Section~\ref{sec:model} is, in particular, an endpoint Uvarov--Legendre family expressed in the Stieltjes basis.  This resolves the structural objection raised against the earlier version of the manuscript.

The nontrivial object is the functional itself.  The master formula
\[
 \mathcal L(t^n)=\frac{g_{x_0,n+1}(b)-g_{x_0,n+1}(a)}{n+1}
\]
and the Stieltjes exponential lead to the exact representation
\[
 \mathcal L(e^{zt})
 =\frac1z\left(e^{c_+z}\Pi_A(z)-e^{-c_-z}\Pi_B(z)^{-1}\right).
\]
This representation has several consequences that are invisible in the original variable.  The detailed continuous part of \(g\) is compressed to the two increments \(c_-\) and \(c_+\); jump locations disappear completely; and jump sizes survive only as the multisets \(A\) and \(B\).  Theorem~\ref{thm:fibres} proves the converse: these four objects are recoverable from the meromorphic moment generating function, so they classify the fibres exactly.

The left--right asymmetry is structural.  Right jumps create the factor \(\Pi_A\) and therefore endpoint derivative evaluations.  For one jump this is the classical Uvarov mass; for two or more jumps it becomes a special Dirac-derivative perturbation whose coefficients are elementary symmetric functions.  Left jumps create \(\Pi_B^{-1}\) and therefore an absolutely continuous tail.  For one left jump the tail is exponential; for finitely many jumps it is the survival function of a finite sum of independent exponentials (hypoexponential in the usual finite setting), while the general theorem also permits a countable summable family of jumps.  Thus the derivative series that appears if one expands about the left endpoint is not best interpreted as an exotic infinite-order object: on polynomials it is the moment expansion of a positive measure.

The positivity theorem is sharp.  A member of the normalized unweighted image is a positive-measure moment functional if and only if there is at most one right jump.  Positive definiteness requires, in addition, infinite support.  The proof uses the Hamburger moment theorem, exponential-moment control and the boundedness of Fourier transforms of finite signed measures; it does not rely on any finite Hankel truncation.  This distinction matters numerically, since a two-right-jump functional with a very small second jump can have many initially positive Hankel determinants before positivity eventually fails.

The weighted lattice construction is deliberately kept separate from the unweighted image theorem.  There the transported moments are factorial moments.  This explains the three qualitatively different examples: Meixner gives a scaled Gamma--Laguerre functional, Charlier gives a rank-one evaluation functional, and Krawtchouk gives the finite derivative functional \eqref{eq:krawfunctional}.  In the last case the exact Hankel formula \eqref{eq:krawdet} shows that the first singular determinant is \(\widehat\Delta_{N+2}\), not \(\widehat\Delta_{N+1}\); finite support of the original discrete weight does not turn the transported functional into the same finite positive measure.

The classification isolates a natural positive family that deserves separate study.  Taking \(A=\emptyset\), \(B\ne\emptyset\), and \(c_-+c_+>0\) gives the weight
\begin{equation}\label{eq:openweight}
 w_B(t)=
 \begin{cases}
 S_B(-c_--t),&t<-c_-,\\
 1,&-c_-\le t\le c_+,\\
 0,&t>c_+.
 \end{cases}
\end{equation}
For a single \(\beta\in B\), the left branch is \(e^{(t+c_-)/\beta}\).  We have found literature on other discontinuous or jump-modified weights, including Gaussian examples \cite{XuZhao2011,Bogatskiy2016}, but not a systematic treatment of this particular flat-plus-exponential or flat-plus-exponential-convolution family.  We therefore leave the following problem open.

\begin{problem}[Flat-plus-exponential-convolution orthogonality]\label{prob:tail}
For the positive weight \eqref{eq:openweight}, determine the monic orthogonal polynomials and study, as explicitly as possible:
\begin{enumerate}
\item their recurrence coefficients and the dependence of those coefficients on \(c_-\), \(c_+\), and the multiset \(B\);
\item localization and asymptotic distribution of their zeros;
\item strong and ratio asymptotics;
\item whether the piecewise logarithmic derivative of \(w_B\) leads to useful ladder, Pearson-type, Riemann--Hilbert, or spectral relations;
\item limiting transitions when one or several elements of \(B\) tend to zero, coalesce, or diverge under a suitable rescaling.
\end{enumerate}
\end{problem}

A second open direction is quasi-definiteness outside the positive region \(|A|\le1\).  The image theorem fixes the functional explicitly, but a complete description of the zero set of its Hankel determinants for arbitrary \(A\) and \(B\) remains open.  A third direction is to repeat the image analysis for centres outside the orthogonality interval and, separately, for the non-monotone Stieltjes calculus of \cite{FrTo,MaiaTojo}.  These are genuine extensions of the classification proved here and should not be conflated with the already resolved algebraic transfer problem.

\section*{Conflicts of interest}

This work does not have any conflicts of interest.

\section*{Data availability}

There is no data related to this manuscript.

\section*{Acknowledgements}

F. J. Fern\'andez and F. A. F. Tojo were supported by grant PID2020-113275GB-I00 funded by Xunta de Galicia, Spain, project ED431C 2023/12; and by MCIN/AEI/\allowbreak 10.13039/501100011033, Spain, and by ``ERDF A way of making Europe'' of the European Union.


\end{document}